\documentclass{article}
\usepackage[utf8]{inputenc}

\usepackage[margin=3cm]{geometry}

\usepackage[style=alphabetic,backend=biber]{biblatex}
\bibliography{references}
\renewbibmacro{in:}{}

\usepackage{color}
\usepackage[color,matrix,arrow]{xy}
\input xy
\xyoption{all}
\usepackage{rotating}

\usepackage{tikz-cd}

\usepackage{amssymb}
\usepackage{amsmath}
\usepackage{mathtools}
\usepackage{faktor}
\usepackage{mathrsfs}

\usepackage{amsthm}
\theoremstyle{definition} \newtheorem{theorem}{Theorem}[section]
\theoremstyle{definition} \newtheorem{definition}[theorem]{Definition}
\theoremstyle{definition} \newtheorem{lemma}[theorem]{Lemma}
\theoremstyle{definition} 
\theoremstyle{definition} \newtheorem{corollary}[theorem]{Corollary}
\theoremstyle{remark} 
\theoremstyle{remark} \newtheorem{remark}[theorem]{Remark}
\theoremstyle{remark} \newtheorem{example}[theorem]{Example}
\theoremstyle{remark} \newtheorem{construction}[theorem]{Construction}

\newcommand{\Zpr}{\mathbb{Z}/p^r}
\newcommand{\Zp}{\mathbb{Z}/p}
\newcommand{\PhiPi}{\Phi_{\nu}^\pi}
\newcommand{\PhiK}{\Phi_{\nu}^K}
\newcommand{\de}{\overline{e}}
\newcommand{\qmax}{q_{\mathrm{max}}}
\newcommand{\qmin}{q_{\mathrm{min}}}
\newcommand{\Jdec}{\Gamma}

\title{$p$-hyperbolicity of homotopy groups via $K$-theory}
\author{Guy Boyde}

\begin{document}

\maketitle

\begin{abstract} We show that $S^n \vee S^m$ is $\Zpr$-hyperbolic for all primes $p$ and all $r \in \mathbb{Z}^+$, provided $n,m \geq 2$, and consequently that various spaces containing $S^n \vee S^m$ as a $p$-local retract are $\Zpr$-hyperbolic. We then give a $K$-theory criterion for a suspension $\Sigma X$ to be $p$-hyperbolic, and use it to deduce that the suspension of a complex Grassmannian $\Sigma Gr_{k,n}$ is $p$-hyperbolic for all odd primes $p$ when $n \geq 3$ and $0<k<n$. We obtain similar results for some related spaces.
    
\end{abstract}

\section{Introduction}

A space $X$ is called \emph{rationally elliptic} if $\pi_*(X) \otimes \mathbb{Q}$ is finite dimensional, and \emph{rationally hyperbolic} if the dimension of $\bigoplus_{i \leq m} \pi_i(X) \otimes \mathbb{Q}$ grows exponentially in $m$. It was proved in \cite[Chapter 33]{RHT} that simply connected $CW$-complexes with rational homology of finite type and finite rational category are either rationally elliptic or rationally hyperbolic. In order to study the $p$-torsion analogue of this dichotomy, Huang and Wu \cite{HuangWu} introduced the definitions of $\Zpr$- and $p$-hyperbolicity.

For $p$ prime, by a \emph{$p$-torsion summand} in an abelian group $A$, we mean a direct summand isomorphic to $\Zpr$ for some $r \geq 1$.


\begin{definition} \label{pdef} Let $X$ be a space, and let $p$ be a prime. We say that $X$ is $p$-\textit{hyperbolic} if the number of $p$-torsion summands in $\pi_*(X)$ grows exponentially, in the sense that $$\liminf_m\frac{\ln(T_m)}{m} > 0,$$ where $T_m$ is the number of $p$-torsion summands in $\bigoplus_{i \leq m} \pi_i(X)$.
\end{definition}

The above definition counts $\Zpr$-summands for all values of $r$. It is also possible to consider only a single $r$, and by doing so we obtain the definition of $\Zpr$-hyperbolicity.

\begin{definition} \label{zprdef} Let $X$ be a space, let $p$ be a prime, and fix $r \in \mathbb{Z}^+$. We say that $X$ is $\Zpr$-\textit{hyperbolic} if the number of $\Zpr$-summands in $\pi_*(X)$ grows exponentially, in the sense that $$\liminf_m\frac{\ln(t_m)}{m} > 0,$$ where $t_m$ is the number of $\Zpr$-summands in $\bigoplus_{i \leq m} \pi_i(X)$. 
\end{definition}

Note that $\Zpr$-hyperbolicity for any $r$ implies $p$-hyperbolicity. It follows immediately from a result of Henn \cite[Corollary of Theorem 1]{HennGrowth} that the $\liminf$s appearing in the above definitions must be finite if $X$ is a simply connected finite $CW$-complex.

Huang and Wu show that for $n \geq 3$, $r \geq 1$ and $p$ any prime, the Moore space $P^n(p^r)$ is $\Zpr$-hyperbolic and $\mathbb{Z}/p^{r+1}$-hyperbolic, and that $P^n(2)$ is also $\mathbb{Z}/8$-hyperbolic \cite[Theorem 1.6]{HuangWu}. More generally, they give criteria in terms of a functorial loop space decomposition due to Selick and Wu  \cite{SelickWu1,SelickWu2} for a suspension $\Sigma X$ to be $\Zpr$-hyperbolic.

More recently, Zhu and Pan \cite{ZhuPan} use a classification of $(n-1)$-connected $CW$-complexes of dimension at most $n+2$, due to Chang \cite{Chang}, to show that, for $n \geq 4$, such a complex is $\Zp$-hyperbolic, provided that it is not contractible or a sphere after $p$-localization. They also prove hyperbolicity results for several so-called elementary Chang complexes.

This paper studies $p$- and $\Zpr$-hyperbolicity of certain suspensions. Our first result is as follows.

\begin{theorem} \label{sphWedge} Let $q_1, q_2 \geq 1$. Then $S^{q_1+1} \vee S^{q_2+1}$ is $\Zpr$-hyperbolic for all primes $p$ and all $r \in \mathbb{Z}^+$.
\end{theorem}

Let $p$ be a prime. If a space $X$ contains a wedge of two spheres as a $p$-local retract, then Theorem \ref{sphWedge} implies that $X$ is $\Zpr$-hyperbolic for all $r$. Various spaces have been shown to have such a wedge as a retract - examples of this sort are given in Section \ref{1.3apps}. A summary is as follows:

\begin{itemize}
    \item for $n,k \geq 3$, the configuration space $\mathrm{Conf}_k(\mathbb{R}^n)$ is $\Zpr$-hyperbolic for all $p$ and $r$ (Example \ref{ConfigurationSpaces});
    \item an $(n-1)$-connected $2n$-dimensional manifold $M$, where $H^n(M)$ is of rank at least 3, is $\Zpr$-hyperbolic for all $p$ and $r$ (Example \ref{niceManifolds});
    \item a generalized moment-angle complex on a simplicial complex having two minimal missing faces which are not disjoint is $\Zpr$-hyperbolic for all $p$ and $r$ (Example \ref{polyhedralProducts});
    \item $\Sigma \mathbb{C}P^2$ is $\Zpr$-hyperbolic for all $p \neq 2$ and all $r$, and $\Sigma \mathbb{H}P^2$ is is $\Zpr$-hyperbolic for all $p \neq 2,3$ and all $r$ (Example \ref{CP2andHP2}).
\end{itemize} 

Our other result is as follows.

\begin{theorem} \label{K-detection} Let $p$ be an odd prime, and let $X$ be a path connected space having the $p$-local homotopy type of a finite $CW$-complex. Suppose that there exists a map $$\mu_1 \vee \mu_2 : S^{q_1+1} \vee S^{q_2+1} \to \Sigma X$$
with $q_i \geq 1$, such that the map $$\widetilde{K}^*(\Sigma X) \otimes \Zp \xrightarrow{(\mu_1 \vee \mu_2)^*} \widetilde{K}^*(S^{q_1 +1} \vee S^{q_2 +1}) \otimes \Zp \cong \Zp \oplus \Zp$$
is a surjection. Then $\Sigma X$ is $p$-hyperbolic.
\end{theorem}

This criterion is quite different to that given by Huang and Wu \cite[Theorem 1.5]{HuangWu}. Their criterion is homotopical, using hypotheses on $X$ to produce retracts of $\Omega \Sigma X$, whereas ours is cohomological, which makes it easier to check. On the other hand, their criterion is stronger, since it gives $\Zpr$-hyperbolicity, rather than just $p$-hyperbolicity. The examples they give, primarily various Moore spaces, differ from those we obtain, which are the suspensions of spaces related to complex projective space. More precisely, in Section \ref{GrassEtc}, we show that the following spaces are $p$-hyperbolic for all $p \neq 2$:

\begin{itemize}
    \item suspended complex projective space $\Sigma \mathbb{CP}^n$ for $n \geq 2$ (Example \ref{CPN}), and more generally;
    \item the suspended complex Grassmannian $\Sigma \mathrm{Gr}_{k,n}$ for $n \geq 3$ and $0<k<n$ (Example \ref{Grassmannians});
    \item the suspended Milnor Hypersurface $\Sigma H_{m,n}$ for $m \geq 2$ and $n \geq 3$, (Example \ref{MilnorHypersurfaces});
    \item the suspended unitary group $\Sigma U(n)$ for $n \geq 3$  (Example \ref{UnitaryGroups}).
\end{itemize}  

Both Theorem \ref{sphWedge} and Theorem \ref{K-detection} will be proven by constructing an exponentially growing family of classes which generate summands in the relevant homotopy groups. We think of this family as `witnessing' the hyperbolicity. For Theorem \ref{sphWedge}, one can proceed directly from the Hilton-Milnor decomposition of $S^n \vee S^m$ \cite{Hilton}. For Theorem \ref{K-detection}, we employ $K$-theoretic methods originally used by Selick \cite{Selick} to prove one direction of Moore's conjecture for suspensions having torsion-free homology.

If the map $\mu_1 \vee \mu_2$ of Theorem \ref{K-detection} induces a surjection on $\widetilde{K}^*(\ \ ) \otimes \Zp$, then so does its suspension $\Sigma \mu_1 \vee \Sigma \mu_2$. The conclusion of Theorem \ref{K-detection} may therefore be strengthened in the following way.

\begin{corollary} \label{allSuspensions} With the hypothesis of Theorem \ref{K-detection}, $\Sigma^n X$ is $p$-hyperbolic for all $n \geq 1$. \qed
\end{corollary}

One might be motivated by this observation to ask whether, in the circumstances of Theorem \ref{K-detection}, the stable homotopy groups of $X$ satisfy the growth conditions of Definition \ref{pdef} or \ref{zprdef}. In the proofs of both Theorem \ref{sphWedge} and \ref{K-detection}, the classes that witness the hyperbolicity are composites involving Whitehead products. The suspension of a Whitehead product is always trivial \cite[Theorem 3.11]{WhiteheadSuspensions}, so the classes we detect cannot be stable. Therefore, Corollary \ref{allSuspensions} does not suggest that the stable homotopy of $\Sigma X$ should be $p$- or $\Zpr$-hyperbolic. On the other hand, it follows from our methods that, under the hypotheses of Theorem \ref{K-detection}, the stable homotopy of $\Omega \Sigma X$ is $p$-hyperbolic.

By a result of Henn \cite{HenncoH}, any co-$H$ space, and in particular any suspension, decomposes rationally as a wedge of spheres. It then follows from the Hilton-Milnor theorem \cite{Hilton} and the computation of the rational homotopy groups of spheres \cite{Serre} that such a suspension is rationally hyperbolic precisely when there are at least two spheres (of dimension $\geq 2$) in this decomposition.

If $\Sigma X$ satisfies the hypotheses of Theorem \ref{K-detection} for any prime (including 2), then by Chern Character considerations the reduced rational homology of $\Sigma X$ has dimension at least two, so $\Sigma X$ is rationally a wedge of at least two spheres by the preceding discussion. This rational equivalence is a local equivalence at all but perhaps finitely many primes, so by Theorem \ref{sphWedge}, $\Sigma X$ is $\Zpr$ hyperbolic for all $r$ at all but finitely many primes $p$. One might therefore conjecture that the conclusion of Theorem \ref{K-detection} can be strengthened to give $\Zpr$-hyperbolicity for all $r$ rather than $p$-hyperbolicity, but we do not know whether this is possible.

We now discuss situations in which it is adequate to consider ordinary cohomology, rather than $K$-theory. If $\Sigma X$ has torsion-free integral (co)homology, or if its cohomology is concentrated in even degrees, then the Atiyah-Hirzebruch spectral sequence for $K^*(\Sigma X)$ collapses on the $E^2$ page \cite{Bundlebook}. It follows by naturality that the image of the map induced by $\mu_1 \vee \mu_2 : S^{q_1+1} \vee S^{q_2+1} \to \Sigma X$ on $K$-theory is identified with the image of the induced map on cohomology. We may therefore replace $K$-theory with cohomology in Theorem \ref{K-detection}, as follows.

\begin{corollary} Let $X$ be a path connected space having the homotopy type of a finite $CW$-complex, such that the Atiyah-Hirzebruch spectral sequence for $K^*(\Sigma X)$ collapses on the $E^2$ page. Let $p$ be an odd prime. Suppose that there exists a map $\mu_1 \vee \mu_2 : S^{q_1+1} \vee S^{q_2+1} \to \Sigma X$
with $q_i \geq 1$, such that the map induced by $\mu_1 \vee \mu_2$ on $\widetilde{H}^*(\ \ ) \otimes \Zp$ is a surjection. Then $\Sigma X$ is $p$-hyperbolic. \label{H-detection}
\end{corollary}

One advantage of ordinary cohomology is that it is connected to the homotopy groups integrally, via the universal coefficient theorem and Hurewicz map. We can exploit this as follows.

\begin{example} Suppose that the Atiyah-Hirzebruch spectral sequence for $K^*(\Sigma X)$ collapses (for example, if $\Sigma X$ has torsion-free homology) and that there exists $q \in \mathbb{Z}^+$ so that $\widetilde{H}_i(\Sigma X)=0$ for $i \leq q$, and $\dim_\mathbb{Q}(\widetilde{H}_{q+1}(\Sigma X) \otimes \mathbb{Q}) \geq 2$. The Hurewicz map $\pi_{q+1}(\Sigma X) \to \widetilde{H}_{q+1}(\Sigma X)$ is an isomorphism, so there exists a map $\mu_1 \vee \mu_2 : S^{q+1} \vee S^{q+1} \to \Sigma X$ inducing the inclusion of a $\mathbb{Z}^2$-summand in $\widetilde{H}_{q+1}(\Sigma X)$. By the universal coefficient theorem relating ordinary homology and cohomology, $\mu_1 \vee \mu_2$ induces a surjection on integral cohomology, so by Corollary \ref{H-detection}, $\Sigma X$ is $p$-hyperbolic for all odd primes $p$.
\end{example}

The structure of this paper is as follows. In Section \ref{applicationSection} we give applications of the main theorems and derive a simple lower bound for the growth of the number of $\Zpr$-summands in the homotopy groups of a wedge of two spheres (Corollary \ref{bound}). The proofs of Theorems \ref{sphWedge} and Theorem \ref{K-detection} may be read largely independently; Section \ref{prelims} contains those preliminary results which are used in both cases. In Section \ref{wedgeSection}, we give the proof of Theorem \ref{sphWedge}. The remainder of the paper is devoted to proving Theorem \ref{K-detection}: Sections \ref{KFacts} and \ref{psiModSection} give the necessary background, and Section \ref{sectionMain} contains the proof. An overview of the proof strategy can be found at the start of Section \ref{sectionMain}.

\medskip

I would like to thank my PhD supervisor, Stephen Theriault, for suggesting the problems that this paper tries to address, and for many helpful conversations along the way. From a technical point of view, much is owed to papers of Huang and Wu \cite{HuangWu}, and of Selick \cite{Selick}. Neil Strickland's `Bestiary' was extremely helpful in providing examples of Theorem \ref{K-detection}. Conversations with Sam Hughes were very useful in formulating Corollary \ref{H-detection}. Thanks to the anonymous reviewer for their thoughtful comments, especially for pointing out that the space in Theorem \ref{K-detection} need only be finite after $p$-localisation.

\section{Applications} \label{applicationSection}

\subsection{Spaces having a wedge of two spheres as a retract} \label{1.3apps}

Theorem \ref{sphWedge} immediately implies that any space $X$ which has $S^{q_1+1} \vee S^{q_2+1}$ as a retract after $p$-localization is $\Zpr$-hyperbolic for that $p$ and all $r$. This implies that for all $n \geq 1$, $\Sigma^n X$ contains $S^{q_1+n+1} \vee S^{q_2+n+1}$ as a $p$-local retract, and so is $\Zpr$-hyperbolic for all $r$. We first consider examples of this form. 

\begin{example}
It is known \cite[Section 3.1]{Knudsen} that $\mathrm{Conf}_k(\mathbb{R}^n)$, the ordered configuration space of $k$ points in $\mathbb{R}^n$, contains $\bigvee_{k-1} S^{n-1}$ as a retract. It follows that, when $n, k \geq 3$, $\mathrm{Conf}_k(\mathbb{R}^n)$ is $\Zpr$-hyperbolic for all $p$ and $r$. \label{ConfigurationSpaces}
\end{example}

\begin{example}Let $M$ be an $(n-1)$-connected $2n$-dimensional manifold. By the universal coefficient theorem, there can be no torsion in $H^n(M)$. Suppose that the rank of $H^n(M)$ is at least 3. By work of Beben and Theriault \cite[Theorem 1.4]{BebenTheriault}, $\Omega M$ contains a wedge of two spheres as a retract after looping. Thus, $M$ is again $\Zpr$-hyperbolic for all $p$ and $r$. \label{niceManifolds}
\end{example}

\begin{example} Let $K$ be a simplicial complex on the vertex set $[m]=\{1, \dots , m\}$, and let $(\underline{X}, \underline{A})$ be any sequence of pairs $(D^{n_i}, S^{n_i-1})$ with $n_i \geq 2$ for $1 \leq i \leq m$. If there exist two distinct minimal missing faces of $K$ which are not disjoint, then by work of Hao, Sun and Theriault \cite[Theorem 4.2]{HaoSunTheriault} the polyhedral product $(\underline{X}, \underline{A})^K$ contains a wedge of two spheres as a retract after looping, and hence is $\Zpr$-hyperbolic for all $p$ and all $r$. \label{polyhedralProducts}
\end{example}

\begin{example} Localized away from 2, $\Sigma \mathbb{C}P^2 \simeq S^3 \vee S^5$. To see this, note that $\Sigma \mathbb{C}P^2$ has a $CW$-structure consisting of one 3-cell and one 5-cell, and that $\pi_4(S^3) \cong \mathbb{Z}/2$ \cite{Freudenthal}. This implies that the attaching map for the 5-cell is nullhomotopic after localization at an odd prime. Thus, $\Sigma \mathbb{C}P^2$ is $\Zpr$-hyperbolic for all $r$ when $p \neq 2$.

Similarly,  $\Sigma \mathbb{H}P^2$ admits a cell structure with one 5-cell and one 9-cell, and $\pi_8(S^5) \cong \mathbb{Z}/24$. Thus, $\Sigma \mathbb{H}P^2$ is $\Zpr$-hyperbolic for all $r$ when $p \neq 2,3$. \label{CP2andHP2}
\end{example}


\subsection{Suspensions of spaces related to $\mathbb{C}P^n$} \label{GrassEtc}

Suppose that one has verified the hypotheses of Theorem \ref{K-detection} for a given space $X$ and odd prime $p$, using a map $\mu_1 \vee \mu_2 : S^{q_1+1} \vee S^{q_2+1} \to \Sigma X$. If another space $Y$ admits a map $\sigma: \Sigma X \to \Sigma Y$ which induces a surjection on $\widetilde{K}^*(\ \ ) \otimes \Zp$, then it is immediate that $\sigma \circ (\mu_1 \vee \mu_2)$ satisfies the hypotheses of Theorem \ref{K-detection}, and hence that $\Sigma Y$ is $p$-hyperbolic. The slogan is that $K$-theory surjections allow us to generate new examples from old ones.

In this section, we will apply this idea. We have seen in Example \ref{CP2andHP2} that, localized away from 2, $\Sigma \mathbb{C}P^2 \simeq S^3 \vee S^5$, so certainly $\Sigma \mathbb{C}P^2$ satisfies the hypotheses of Theorem \ref{K-detection} at all odd primes $p$. We will now consider spaces $X$ which are known to admit maps $\mathbb{C}P^2 \to X$ which induce surjections on integral $K$-theory, and hence on $\widetilde{K}^*(\ \ ) \otimes \Zp$ for all odd $p$. It follows in each case that $\Sigma X$ is $p$-hyperbolic, and further by Corollary \ref{allSuspensions}, that $\Sigma^n X$ is $p$-hyperbolic for all $n \geq 1$.

The inclusion of $\mathbb{C}P^n$ into $\mathbb{C}P^{n+1}$ induces a surjection on $K$-theory, so it must still induce a surjection after suspending. Composing these inclusions with the local equivalence $\Sigma \mathbb{C}P^2 \simeq S^3 \vee S^5$ gives, for each $n \geq 2$, a map $S^3 \vee S^5 \to \Sigma \mathbb{C}P^n$ which still induces a surjection on $\widetilde{K}^*(\ \ ) \otimes \Zp$ for all odd primes $p$. Applying Theorem \ref{K-detection} to this map gives the following.

\begin{example} \label{CPN} For $n \geq 2$, $\Sigma \mathbb{C} P^n$ is $p$-hyperbolic for all $p \neq 2$.
\end{example}

Now let $\mathrm{Gr}_{k,n}$ be the Grassmannian of $k$-dimensional complex subspaces of $\mathbb{C}^n$. First note that orthogonal complement gives a homeomorphism $\mathrm{Gr}_{k,n} \cong \mathrm{Gr}_{n-k,n}$. In particular $\mathrm{Gr}_{n-1,n} \cong \mathrm{Gr}_{1,n} \cong \mathbb{C}P^{n-1}$, so $\Sigma \mathrm{Gr}_{n-1,n}$ is $p$-hyperbolic. Other Grassmannians can be treated more uniformly, as follows.

Let $\gamma_{k,n}$ denote the tautological bundle over $\mathrm{Gr}_{k,n}$. Consider the inclusion $$\iota_n : \mathbb{C}^n \to \mathbb{C}^{n+1}$$
$$(x_1, x_2, \dots, x_n) \mapsto (x_1, x_2, \dots, x_n, 0).$$

This inclusion induces a map $i_{k,n} : \mathrm{Gr}_{k,n} \to \mathrm{Gr}_{k,n+1}$, defined on $V \in \mathrm{Gr}_{k,n}$ by $V \mapsto \iota_n (V)$. It follows from this definition that $i_{k,n}^*(\gamma_{k,n+1}) = \gamma_{k,n}$. Letting $e_i$ denote the $i$-th standard basis vector in $\mathbb{C}^n$, we also have a map $j_{k,n} : \mathrm{Gr}_{k,n} \to \mathrm{Gr}_{k+1,n+1}$, defined on $V = \mathrm{Span}(v_1, v_2, \dots, v_k) \in \mathrm{Gr}_{k,n}$ by $V \mapsto \mathrm{Span}(\iota (v_1), \iota(v_2), \dots, \iota(v_k), e_{n+1})$. It follows from this definition that $j_{k,n}^*(\gamma_{k+1,n+1}) = \gamma_{k,n} \oplus \underline{\mathbb{C}^1}$, where $\underline{\mathbb{C}^1}$ is the 1-dimensional trivial bundle.

Since $K^*(\mathbb{C}P^n)$ is generated by the class of the tautological bundle, composing the maps $i_{k,n}$ and $j_{k,n}$ for different values of $k$ and $n$ will give maps $\mathbb{C}P^2 = \mathrm{Gr}_{1,3} \to \mathrm{Gr}_{k,n}$ for all $1 \leq k \leq n-2$ and $n \geq 3$ which induce surjections in integral $K$-theory. As in Example \ref{CPN}, this implies the following (the case $k=n-1$ is $\mathrm{Gr}_{n-1,n}$, which was treated first).

\begin{example} \label{Grassmannians} For $n \geq 3$ and $0<k<n$, the suspended complex Grassmannian $\Sigma \mathrm{Gr}_{k,n}$ is $p$-hyperbolic for all $p \neq 2$.
\end{example}

For $m \leq n$, the \emph{Milnor Hypersurface} $H_{m,n}$  is defined by $$H_{m,n}=\{ ([z],[w]) \in \mathbb{C} P^m \times \mathbb{C} P^n \ | \ \sum_{i=0}^m z_i w_i = 0\}.$$

Suppose that $m \geq 2$ and $n \geq 3$. Then there is an inclusion $\iota: \mathbb{C}P^2 \to H_{m,n}$, defined by
$$\iota([z_0: z_1: z_2]) = ([z_0:z_1:z_2:0:\dots:0],[0:\dots:0:1]).$$

Write $\pi_1$ for the projection $H_{m,n} \to \mathbb{C} P^m$. Then the inclusion $\mathbb{C}P^2 \to \mathbb{C}P^m$ factors as

\begin{center}
\begin{tabular}{c}
\xymatrix{
\mathbb{C}P^2 \ar[dr] \ar[r]^{\iota} & H_{m,n} \ar[d]^{\pi_1} \\
& \mathbb{C}P^m.
}
\end{tabular}
\end{center}

This implies that $\iota$ induces a surjection on integral $K$-theory, so we obtain the following.

\begin{example} \label{MilnorHypersurfaces} For $m \geq 2$ and $n \geq 3$, the suspended Milnor Hypersurface $\Sigma H_{m,n}$ is $p$-hyperbolic for all $p \neq 2$.
\end{example}

Let $U(n)$ denote the unitary group. There is a well-known map $r:\Sigma \mathbb{C}P^{n-1} \to U(n)$ (see for example \cite{WhiteheadBook}) which induces a surjection on $K$-theory. From this we obtain

\begin{example} \label{UnitaryGroups} For $n \geq 3$, the suspended unitary group $\Sigma U(n)$ is $p$-hyperbolic for all $p \neq 2$. 
\end{example}

\subsection{Quantitative lower bounds on growth} \label{quantitativeLBs}

In Section \ref{wedgeSection}, we will derive the following simple lower bound for the $\liminf$ in the definition of $\Zpr$-hyperbolicity, for the space $S^{q_1+1} \vee S^{q_2+1}$.

\begin{corollary} \label{bound} Let $p$ be a prime and $r \in \mathbb{Z}^+$. Let $t_m$ be the constants of Definition \ref{zprdef} for $X=S^{q_1+1} \vee S^{q_1+1}$. Then $$\liminf_m\frac{\ln(t_m)}{m} \geq \frac{\ln(2)}{\mathrm{max}(q_1,q_2)}.$$
\end{corollary}

This implies that $t_m$ eventually exceeds $((1-\varepsilon)2)^{\frac{m}{\mathrm{max}(q_1,q_2)}}$ for any $\varepsilon >0$. The constant 2 reflects the number of wedge summands. Note that this bound is independent of $p$ and $r$.

\begin{example} Taking $\varepsilon = \frac{1}{4}$, we find that for all $r \in \mathbb{Z}^+$ and all primes $p$ the number of $\Zpr$-summands in $\bigoplus_{i \leq m} \pi_i(S^2 \vee S^2)$ eventually exceeds $(\frac{3}{2})^{m}$.
\end{example}

One can produce an analogous quantitative bound on the $\liminf$ in the case of Theorem \ref{K-detection}, but this bound is very weak. In particular, it depends on knowledge of the Adams operations on $K^*(X)$, and is at best $\frac{\ln(2)}{2(p-1)}$.

\section{Preliminary results} \label{prelims}

Both Theorem \ref{sphWedge} and Theorem \ref{K-detection} will be proven by means of Lemma \ref{linearimplieshype}. Our first goal is to establish this lemma.

Let $L$ be the free Lie algebra over $\mathbb{Q}$ on basis elements $x_1 , \dots, x_n$. Write $\mathscr{L}_k$ for the subset of $L$ consisting of the \emph{basic products} of the $x_i$ of \emph{weight} $k$, in the sense of \cite{Hilton}, where the basic products of weight 1 are taken to be the $x_i$, ordered by $x_1 < x_2 < \dots < x_n$. The union $\mathscr{L} = \bigcup_{k=1}^\infty \mathscr{L}_k$ is a vector space basis for $L$ (see for example \cite[Theorem 5.3]{SerreBook}, but note that Serre uses the name \emph{Hall basis} for the set of basic products).

Let $\mu: \mathbb{Z}^+ \longrightarrow \{-1, 0, 1\}$ be the M\"obius inversion function, defined by $$\mu(s) = \begin{cases} 
      1 & s=1 \\
      0 & s>1 \textrm{ is not square free} \\
      (-1)^{\ell} & s>1 \textrm{ is a product of $\ell$ distinct primes.}
      \end{cases}$$ The Witt Formula $W_n(k)$ is then defined by

$$W_n(k)=\frac{1}{k} \sum_{d \lvert k} \mu(d) n^{\frac{k}{d}}.$$

\begin{theorem}{\cite[Theorem 3.3]{Hilton}} \label{ungradedWitt} Let $L$ be the free Lie algebra over $\mathbb{Q}$ on basis elements $x_1 , \dots, x_n$. Then $|\mathscr{L}_k| = W_n(k)$. \qed \end{theorem}

\begin{lemma}{\cite[Introduction]{Bahturin}} \label{ratio} The ratio
$$\frac{W_n(k)}{\frac{1}{k}n^k}$$
tends to 1 as $k$ tends to $\infty$. $\square$
\end{lemma}

In particular, this implies that for $n \geq 2$, the Witt function $W_n(k)$ grows exponentially in $k$. It should follow that if the number of $p$-torsion summands in $\bigoplus_{i \leq k} \pi_i(Y)$ exceeds $W_2(k)$, then $Y$ is $p$-hyperbolic. The following lemma makes a slightly generalised form of this idea precise.

\begin{lemma} \label{linearimplieshype} Let $Y$ be a space. Suppose that there exist $a,b \in \mathbb{Z}^+$ such that the number of $p$-torsion summands (respectively, $\Zpr$-summands) in $\bigoplus_{i \leq ak+b} \pi_{i}(Y)$ exceeds $W_2(k)$, for all $k$ large enough. Then $Y$ is $p$-hyperbolic (respectively, $\Zpr$-hyperbolic).
\end{lemma}

\begin{proof} The proofs for $p$- and $\Zpr$-hyperbolicity are identical, so we give only the former. Reframing the hypothesis in terms of the sequence $\{ T_m \}_m$ of Definition \ref{pdef}, we are assuming precisely that $T_{ak+b} > W_2(k)$ for sufficiently large $k$. We then have that

$$\liminf_m\frac{\ln(T_m)}{m} = \liminf_k\frac{\ln(T_{ak+b})}{ak+b} \geq \liminf_k\frac{\ln(W_2(k))}{ak+b}.$$

It then follows from Lemma \ref{ratio} that if $1>\varepsilon >0$, once $k$ is large enough, we have $$W_2(k) > (1-\varepsilon)\frac{1}{k}2^k.$$

This implies that $$\liminf_k\frac{\ln(W_2(k))}{ak+b} \geq \liminf_k\frac{\ln((1-\varepsilon)\frac{1}{k}2^k)}{ak+b},$$ and since this holds for all $\varepsilon >0$, $$\liminf_m\frac{\ln(T_m)}{m} \geq \liminf_k\frac{\ln(W_2(k))}{ak+b} \geq \liminf_k\frac{\ln(\frac{1}{k}2^k)}{ak+b} = \liminf_k\frac{\ln(\frac{1}{k})+k\ln(2)}{ak+b}=\frac{\ln(2)}{a},$$ which is greater than zero, as required.
\end{proof}

\subsection{Existence of summands in the stable stems} 

We write $\pi_j^S$ for the $j$-th stable stem in the homotopy groups of spheres, that is
$$\pi_j^S:=\lim_{n \rightarrow \infty} \pi_{n+j}(S^n).$$

The proof of Theorem \ref{sphWedge}, depends on having, for each $p$ and $r$, some $j$ such that $\pi_j^S$ contains a $\Zpr$-summand. The purpose of this subsection is to show that the existence of such a $j$ follows from existing work of Adams and others.

\begin{lemma} \label{fuel} For any prime $p$ and any $r \in \mathbb{Z}^+$, there exists $j$ such that $\Zpr$ is a direct summand in $\pi^S_j$. That is, for a fixed choice of such a $j$, $\Zpr$ is a direct summand in $\pi_{n+j}(S^n)$ for all $n \geq j+2$.
\end{lemma}

\begin{proof} We write $\nu_p(s)$ for the largest power of $p$ dividing the integer $s$.

CASE 1 ($p$ odd): Set $t:=p^{r-1}(p-1)$, and notice that, since $(p-1)$ is even, $j:=4t-1$ is congruent to $7$ mod 8. Theorem 1.6 of \cite{AdamsIV}, and the discussion immediately following it, then tells us that $\pi_j^S$ contains a direct summand isomorphic to $\mathbb{Z}/m(2t)$, for a function $m$ which Adams defines. By decomposing this subgroup into direct summands of prime power order, it suffices to show that $\nu_p(m(2t))=r$.

The discussion after Theorem 2.5 in \cite{AdamsII} gives that since $t \equiv 0$ mod $(p-1)$, $$\nu_p(m(2t))=1+\nu_p(2t).$$

Now, $\nu_p(2t)$ is equal to $(r-1)$, by definition of $t$, so $\nu_p(m(2t))=r$, as required.

CASE 2 ($p=2$, $r \geq 3$): Set $t:=2^{r-3}$, and set $j:=4t-1$. From Theorem 1.5, and the discussion following Theorem 1.6 in \cite{AdamsIV}, $\pi_j^S$ has a direct summand isomorphic to $\mathbb{Z}/m(2t)$, regardless of whether $j$ is congruent to 3 or 7 mod 8. Again, referring to the discussion after Theorem 2.5 of \cite{AdamsII}, we see that

$$\nu_2(m(2t))=2+\nu_2(2t)=3+\nu_2(t)=r,$$

as required.

CASE 3 ($p^r=2$ and $p^r=4$): It is known from \cite{Freudenthal} that $\pi_1^S \cong \mathbb{Z}/2$, and from \cite{BarrattMahowaldTangora} that $\pi_{34}^S \cong \mathbb{Z}/4 \oplus (\mathbb{Z}/2)^3$.
\end{proof}

\section{Proof of Theorem \ref{sphWedge}} \label{wedgeSection}

In this section we prove Theorem \ref{sphWedge}, which says that the wedge of two spheres is $\Zpr$-hyperbolic for all $p$ and $r$. We also prove Corollary \ref{bound}, which extracts some simple quantitative information from the proof of Theorem \ref{sphWedge}. We first record the following simple observation.

\begin{remark} \label{kqObservation} Let $k_1, \dots, k_n$ and $q_1, \dots q_n$ be non-negative integers. Suppose that $q_1 \leq q_2 \leq \dots \leq q_n$, and let $k = \sum_{i=1}^n k_i$. Then $$k q_1 \leq \sum_{i=1}^n k_i q_i \leq k q_n.$$
\end{remark}

\begin{proof}[Proof of Theorem \ref{sphWedge}] Assume without loss of generality that $q_1 \leq q_2$. By Lemma \ref{linearimplieshype} it suffices to show that there exist constants $a$ and $b$ such that the number of $\Zpr$-summands in $$\bigoplus_{i \leq ak+b} \pi_{i}(S^{q_1+1} \vee S^{q_2+1})$$ exceeds $W_2(k)$, for $k$ large enough.

We first apply the Hilton-Milnor Theorem. Since we are dealing with spheres, we need only the original form, due to Hilton in \cite{Hilton}:

$$\Omega (S^{q_1+1} \vee S^{q_2+1}) \simeq \Omega \Sigma (S^{q_1} \vee S^{q_2}) \simeq \Omega \prod_{B \in \mathscr{L}} S^{k_1 q_1 + k_2 q_2 + 1},$$ where, as in Section \ref{prelims}, $\mathscr{L} = \bigcup_{k=1}^\infty \mathscr{L}_k$ is Hilton's `basic product' basis for $L$, the free Lie Algebra over $\mathbb{Q}$ on two generators $x_1$ and $x_2$, and $k_i$ is the number of occurrences of the generator $x_i$ in the bracket $B$. Recall also from Section \ref{prelims} that the weight $k$ of a bracket $B$ is equal to $k_1+k_2$, and that the cardinality of $\mathscr{L}_k$ is given by the Witt formula $W_2(k)$ by Theorem \ref{ungradedWitt}.

For fixed $k \in \mathbb{Z}^+$, consider the factor in the above product corresponding to $\mathscr{L}_k \subset \mathscr{L}$: $$F_k := \Omega \prod_{B \in \mathscr{L}_k} S^{k_1 q_1 + k_2 q_2 + 1}.$$ The associated subgroup of $\pi_*(S^{q_1+1} \vee S^{q_2+1})$,

$$\bigoplus_{B \in \mathscr{L}_k} \pi_*(S^{k_1 q_1 + k_2 q_2 + 1}),$$ is a direct summand.

We will first find a $\Zpr$-summand in the homotopy groups of each of the spheres appearing in $F_k$. Since $q_1 \leq q_2$, Remark \ref{kqObservation} applies, and we may lower bound the dimensions of spheres appearing in $F_k$ by $k_1 q_1 + k_2 q_2 + 1 \geq k q_1+1$. By Lemma \ref{fuel}, there exists $j \in \mathbb{Z}^+$ such that $\pi_{j+\ell}(S^\ell)$ has a direct summand $\Zpr$ for $\ell \geq j+2$. Therefore, if $k$ is large enough that $k q_1 \geq j+1$, then $k_1 q_1 + k_2 q_2 + 1 \geq j+2$ - that is, the $j$-th stem is stable on all of the spheres occurring in $F_k$. Thus, for $k$ large enough, there is a $\Zpr$ summand in $\pi_{j+k_1 q_1 + k_2 q_2 + 1}(S^{k_1 q_1 + k_2 q_2 + 1})$ whenever $k_1+k_2=k$.

We now upper bound the dimension of the homotopy groups in which these summands appear. Since $q_1 \leq q_2$ we have by Remark \ref{kqObservation} that $j+k_1 q_1 + k_2 q_2 + 1 \leq k q_2+1+j$, so each of the $\Zpr$-summands we have identified is a distinct direct summand in

$$\bigoplus_{i \leq k q_2+1+j} \ \bigoplus_{B \in \mathscr{L}_k} \pi_i(S^{k_1 q_1 + k_2 q_2 + 1}),$$ hence in $$\bigoplus_{i \leq k q_2+1+j} \pi_i(S^{q_1+1} \vee S^{q_2+1}).$$ We have identified one such summand for each $B \in \mathscr{L}_k$, so the number of $\Zpr$-summands in $\bigoplus_{i \leq k q_2+1+j} \pi_i(S^{q_1+1} \vee S^{q_2+1})$ is at least $|\mathscr{L}_k| = W_2(k)$. Thus, taking $a=q_2$ and $b=1+j$ in Lemma \ref{linearimplieshype} suffices. \end{proof}

\begin{proof}[Proof of Corollary \ref{bound}]
The last line of the proof of Lemma \ref{linearimplieshype} shows that $\liminf_m \frac{\ln{t_m}}{m} > \frac{\ln{2}}{a}$. The last line of the proof of Theorem \ref{sphWedge} implies that $a$ may be taken to be $q_2$, under the assumption that $q_1 \leq q_2$, which implies the result.
\end{proof}

\section{$K$-theory and $K$-homology of $\Omega \Sigma X$} \label{KFacts}

The remainder of this paper proves Theorem \ref{K-detection}. Sections \ref{KFacts} and \ref{psiModSection} assemble necessary background, which we will use in Section \ref{sectionMain} to prove the result.

When studying the homotopy groups of a suspension $\Sigma X$, as in Theorem \ref{K-detection}, the following approach is natural. Since $\pi_*(\Sigma X) \cong \pi_{*-1}(\Omega \Sigma X)$, we may instead study $\Omega \Sigma X$. This is useful because $\Omega \Sigma X$ is well understood homologically via the Bott-Samelson theorem, which decomposes its homology as the tensor algebra on $\widetilde{H}_*(X)$. Because we will need to use Adams' $e$-invariant, which is defined in terms of $K$-theory, we wish to replace ordinary homology with $K$-homology.

The purpose of Section \ref{KFacts} is to record the version of the Bott-Samelson theorem which applies to (torsion-free) $K$-homology, along with a universal coefficient theorem for passing between $K$-theory and $K$-homology. All of the material here is already known (in particular much of it is in \cite{Selick}) so its summary here is for convenience and clarity.

Our conventions on definition of $\widetilde{K}^*(X)$ are those of \cite{AtiyahHirzebruch}. In particular, we define $\widetilde{K}^{-1}(X) := \widetilde{K}^0(\Sigma X)$, and set $\widetilde{K}^*(X) := \widetilde{K}^0( X) \oplus \widetilde{K}^{-1}( X)$. We regard $\widetilde{K}^*(X)$ and $\widetilde{K}_*(X)$ as being $\mathbb{Z}/2$-graded. It is shown in \cite{AtiyahHirzebruch} that $\widetilde{K}^*(X)$ is a $\mathbb{Z}/2$-graded ring.

We will wish to work with $K$-theory and $K$-homology modulo the torsion subgroup. For a space $X$, write $\widetilde{K}_*^{\mathrm{TF}}(X)$ and $\widetilde{K}^*_{\mathrm{TF}}(X)$ for the quotients of the reduced $K$-homology and $K$-theory of $X$ by their torsion subgroups. The same convention applies in the unreduced case.

\subsection{K\"unneth and universal coefficient theorems}

The universal coefficient theorem for $K$-theory first appears in some unpublished lecture notes of Anderson \cite{Anderson}, and is first published by Yosimura \cite{Yosimura}.

\begin{theorem}[Universal coefficient theorem] \label{originalUCT} For any $CW$-complex $X$ and each integer $n$ there is a short exact sequence
\[ \pushQED{\qed} 
0 \to \mathrm{Ext}(K_{n-1}(X), \mathbb{Z}) \to K^n(X) \to \mathrm{Hom}(K_n(X), \mathbb{Z}) \to 0. \qedhere
\popQED
\]
\end{theorem}

In the torsion-free case, the universal coefficient theorem is as follows, where, unsurprisingly, we write $\mathrm{Ext}(\widetilde{K}_{n-1}(X), \mathbb{Z})^{\mathrm{TF}}$ for the quotient of $\mathrm{Ext}(\widetilde{K}_{n-1}(X), \mathbb{Z})$ by its torsion subgroup.

\begin{corollary} \label{usableUCTs} \begin{enumerate}
    \item For any $CW$-complex $X$ and each integer $n$ there is a short exact sequence $$0 \to \mathrm{Ext}(\widetilde{K}_{n-1}(X), \mathbb{Z})^{\mathrm{TF}} \to \widetilde{K}^n_{\mathrm{TF}}(X) \to \mathrm{Hom}(\widetilde{K}_n^{\mathrm{TF}}(X), \mathbb{Z}) \to 0.$$
    \item If $X$ is a finite $CW$-complex, then $\mathrm{Ext}(\widetilde{K}_{n-1}(X), \mathbb{Z})^{\mathrm{TF}} = 0$, and we obtain an isomorphism $  \widetilde{K}^n_{\mathrm{TF}}(Y) \xrightarrow{\cong} \mathrm{Hom}(\widetilde{K}_n^{\mathrm{TF}}(Y), \mathbb{Z}). $
\end{enumerate}
\end{corollary}

\begin{proof} Let $T_1$ denote the torsion subgroup of $\mathrm{Ext}(\widetilde{K}_{n-1}(X), \mathbb{Z})$, and let $T_2$ be the torsion subgroup of $\widetilde{K}^n(X)$. The Universal Coefficient Sequence of Theorem \ref{originalUCT} gives an injection $T_1 \to \widetilde{K}^n(X)$, which must have image contained in $T_2$, thus lift to an injection $T_1 \to T_2$. For any group $G$, $\mathrm{Hom}(G, \mathbb{Z})$ is torsion-free, so the composite $T_2 \to \widetilde{K}^n(X) \to \mathrm{Hom}(K_n(X), \mathbb{Z})$ is trivial, and by exactness we obtain a lift $T_2 \to \mathrm{Ext}(\widetilde{K}_{n-1}(X), \mathbb{Z})$. The image of this map must be torsion, which is to say that it must be contained in $T_1$, so the aforementioned map $T_1 \to T_2$ is a surjection. That is, the Universal Coefficient Sequence of Theorem \ref{originalUCT} has last term torsion-free and first map restricting to an isomorphism of torsion subgroups. This implies the first statement.

For the second statement, we need only note that if $X$ is finite, then $\widetilde{K}_{n-1}^{\mathrm{TF}}(X)$ is finitely generated, so $\mathrm{Ext}(\widetilde{K}_{n-1}(X), \mathbb{Z})$ is torsion, as required. \end{proof}

Selick \cite{Selick} deduces the following from work of Atiyah \cite{Atiyah}, Mislin \cite{Mislin} and Adams \cite{Adams}.

\begin{theorem}[K\"unneth theorem for K-homology] \label{KKun} Let $X$ and $Y$ be of the homotopy type of finite complexes. Then there is an isomorphism of $\mathbb{Z}/2$-graded  $\mathbb{Z}$-modules:
\[ \pushQED{\qed} 
\widetilde{K}_*^{\mathrm{TF}}(X \wedge Y) \cong \widetilde{K}_*^{\mathrm{TF}}(X) \otimes \widetilde{K}_*^{\mathrm{TF}}(Y). \qedhere
\popQED
\]
\end{theorem}

\begin{remark} \label{sphGenChoice} It follows immediately from Corollary \ref{usableUCTs} (and knowledge of $\widetilde{K}^*(S^q)$) that $\widetilde{K}_*^{\mathrm{TF}}(S^q) \cong \mathbb{Z}$. We write $\xi_q$ for the generator of $\widetilde{K}_*^{\mathrm{TF}}(S^q)$. By the K\"unneth Theorem (Theorem \ref{KKun}), we may choose the $\xi_q$ so that $\xi_n \otimes \xi_m$ is identified with $\xi_{n+m}$ under the homeomorphism $S^n \wedge S^m \cong S^{n+m}$.
\end{remark}

In the case of $K$-theory, the analogous result follows directly from \cite{Adams}.

\begin{theorem}[K\"unneth theorem for K-theory] \label{KUpKun} Let $X$ and $Y$ be of the homotopy type of finite complexes. Then the external product on $K$-theory defines an isomorphism of $\mathbb{Z}/2$-graded rings:
$$ \widetilde{K}^*_{\mathrm{TF}}(X) \otimes \widetilde{K}^*_{\mathrm{TF}}(Y) \xrightarrow{\cong} \widetilde{K}^*_{\mathrm{TF}}(X \wedge Y).$$

\end{theorem}

\subsection{The James construction}

For a space $X$, let $X^s$ denote the product of $s$ copies of $X$. Let $\sim$ be the relation on $X^s$ defined by $$(x_1, \dots ,x_{i-1}, * , x_{i+1}, x_{i+2}, \dots x_{s}) \sim (x_1, \dots , x_{i-1} , x_{i+1}, * , x_{i+2} , \dots x_{s}).$$
Let $J_s(X)$ be the space $\faktor{X^s}{\sim}$. There is a natural inclusion $$J_s(X) \xhookrightarrow{} J_{s+1}(X)$$ $$(x_1, \dots , x_s) \mapsto (x_1, \dots, x_s, *).$$

The \emph{James construction} $JX$ is defined to be the colimit of the diagram consisting of the spaces $J_s(X)$ and the above inclusions. Write $i_s : J_s(X) \to JX$ for the map associated to the colimit. Notice that $JX$ carries a product given by concatenation, which makes it into the free topological monoid on $X$, and that a topological monoid is in particular an $H$-space.

Let $X^{\wedge i}$ denote the smash product of $i$ copies of $X$, and let $\eta : X \to \Omega \Sigma X$ be the unit of the adjunction $\Sigma \dashv \Omega$. Explicitly, $\eta(x) = (t \mapsto \langle x, t \rangle \in \Sigma X)$.

\begin{theorem}{\cite{James}} \label{JamesConstruction}
\begin{enumerate}
    \item There is a homotopy equivalence $JX \xrightarrow{\simeq} \Omega \Sigma X$ which respects the $H$-space structures and identifies $i_1$ with $\eta$. 
    \item There is a homotopy equivalence $ \bigvee_{i=1}^\infty \Sigma X^{\wedge i} \xrightarrow{\simeq} \Sigma JX$ which restricts to a homotopy equivalence $\bigvee_{i=1}^s \Sigma X^{\wedge i} \xrightarrow{\simeq} \Sigma J_s(X)$ for each $s \in \mathbb{Z}^+$. \qed
\end{enumerate}
\end{theorem}

\begin{lemma}{\cite[Lemma 7]{Selick}} \label{jTrunk} Let $X$ have the homotopy type of an $(r-1)$-connected $CW$-complex.

\begin{enumerate}
    \item $(i_s)_* : \pi_N(J_s(X)) \to \pi_N(JX)$ is an isomorphism for $N < r(s+1)-1$.
    \item Let $x \in \pi_N(J_s(X))$ for any $N$. If $\Sigma x$ is nontrivial then $(i_s)_*(x)$ is also nontrivial.
\end{enumerate}
\end{lemma}

\begin{proof} The first part follows by cellular approximation from the observation that $J_s(X)$ contains the $(r(s+1)-1)$-skeleton of $JX$. The second part follows from the observation that $\Sigma i_s$ has a retraction by Theorem \ref{JamesConstruction}. \end{proof}

For spaces $X$ and $Y$, let $X * Y$ denote the \emph{join}, which we define to be the homotopy pushout of the projections $X \times Y \to X$ and $X \times Y \to Y$. The join is naturally a quotient of $X \times I \times Y$, where $I$ denotes the unit interval. Following the treatment in \cite{Arkowitz}, let $C_1$ denote the subspace of $X * Y$ consisting of points of the form $(x, t, *)$, for $t \in I$ and $x \in X$, and let $C_2$ be the subspace consisting of points of the form $(*, t, y)$. The subspace $C_1 \cup C_2 \cong CX \cup CY$ is contractible, so the quotient map $q: X * Y \to \faktor{X * Y}{C_1 \cup C_2}$ is a homotopy equivalence. The quotient $\faktor{X * Y}{C_1 \cup C_2}$ is homeomorphic to $\Sigma X \wedge Y$. The suspended product $\Sigma (X \times Y)$ is also a quotient of $X \times I \times Y$, and this quotient lies between $X * Y$ and $\faktor{X * Y}{C_1 \cup C_2}$.

This gives a factorization of $q$ as $X * Y \to \Sigma(X \times Y) \to \Sigma (X \wedge Y)$. Let $q^{-1}$ denote any choice of homotopy inverse to $q$; all possible choices are homotopic. We may form a new map $\delta_{X,Y}$ as the composite $\Sigma (X \wedge Y) \xrightarrow{q^{-1}} X * Y \to \Sigma(X \times Y)$. It is automatic that $\delta_{X,Y}$ splits the quotient map $\pi: \Sigma (X \times Y) \to \Sigma X \wedge Y$. The homotopy class of $\delta_{X,Y}$ is well-defined, and we will call $\delta_{X,Y}$ the \emph{canonical splitting} of $\pi$. Note that $\delta_{X,Y}$ is natural in maps of spaces in the sense that given $f : A \to X$ and $g: B \to Y$ we obtain a commutative diagram\begin{center}
\begin{tabular}{c}
\xymatrix{
\Sigma A \wedge B \ar[r]^{\delta_{A,B}} \ar[d]^{\Sigma (f \wedge g)} & \Sigma (A \times B) \ar[d]^{\Sigma (f \times g)} \\
\Sigma X \wedge Y \ar[r]^{\delta_{X,Y}} & \Sigma (X \times Y).
}
\end{tabular}
\end{center}

For $s \geq 3$, consider the quotient map $\Sigma X^s \to \Sigma X^{\wedge s}$. We define the \emph{canonical splitting} of this quotient to be the composite of canonical splittings $$\Sigma X^{\wedge s} \to \Sigma (X \times X) \wedge X^{\wedge (s - 2)} \to \Sigma ((X \times X) \times X) \wedge X^{\wedge (s - 3)} \to \dots \to \Sigma X^s.$$ Of course, we chose an order of multiplication here. This canonical splitting is natural as before.

\begin{definition} \label{tensorAlgebraDef} For a $\mathbb{Z}$-graded (respectively $\mathbb{Z}/2$-graded) module $M$, let $T(M) = \bigoplus_{k=1}^\infty M^{\otimes k}$ denote the \emph{tensor algebra} on $M$. The product is given by concatenation. We refer to $M^{\otimes k}$ as the \emph{weight $k$ component} of the tensor algebra $T(M)$. We define a $\mathbb{Z}$-grading (respectively $\mathbb{Z}/2$-grading) on $T(M)$ by setting $|x_1 \otimes x_2 \otimes \dots \otimes x_k| = \sum_{i=1}^k |x_i|$. \end{definition}

\begin{definition} \label{desuspension} For a space $Y$, let $\sigma: \widetilde{K}^\mathrm{TF}_*(Y) \xrightarrow{\cong} \widetilde{K}^\mathrm{TF}_{*+1}(\Sigma Y)$ be the suspension isomorphism. Let $\varphi: \widetilde{K}^\mathrm{TF}_{*}(\Sigma Y) \to \widetilde{K}^\mathrm{TF}_{*}(\Sigma Y)$ be a homomorphism of graded groups, not necessarily induced by a map of spaces. We call the composite $\sigma^{-1} \circ \varphi \circ \sigma$ the \emph{desuspension} of $\theta$, denoting it by $S^{-1} \varphi$. \end{definition}

Write $m_s: (\Omega \Sigma X)^s \to \Omega \Sigma X$ for the map given by iteratively performing the standard loop multiplication on $\Omega \Sigma X$ in any choice of order. Up to homotopy, $m_s$ is independent of this choice of order, since $\Omega \Sigma X$ is homotopy associative. 

Theorem \ref{JamesConstruction} gives the existence of a homotopy equivalence $\Jdec:\bigvee_{i=1}^\infty \Sigma X^{\wedge i} \to \Sigma \Omega \Sigma X$. There are many choices of $\Jdec$, up to homotopy. The next lemma asserts that $\Jdec$ can be chosen in a way which suits our purpose. Selick \cite{Selick} describes the composite $\bigvee_{i=1}^\infty \Sigma X^{\wedge i} \xrightarrow{\Jdec} \Sigma \Omega \Sigma X \xrightarrow{\simeq} \Sigma JX$ of $\Jdec$ with the homotopy equivalence of Theorem \ref{JamesConstruction} (1). This immediately implies the following description of $\Jdec$.

\begin{lemma}{\cite{Selick}} \label{GetTensor} Let $X$ be a finite $CW$-complex. The homotopy equivalence $\Jdec:\bigvee_{i=1}^\infty \Sigma X^{\wedge i} \to \Sigma \Omega \Sigma X$ may be chosen such that: 
\begin{enumerate}
    \item $ S^{-1} (\Jdec_*):T(\widetilde{K}^{\mathrm{TF}}_*(X)) \xrightarrow{\cong} K^{\mathrm{TF}}_*(\Omega \Sigma X)$ is an isomorphism of algebras;
    \item the restriction of $\Jdec$ to $\Sigma X^{\wedge s}$ is homotopic to the composite $$
\Sigma X^{\wedge s} \to \Sigma X^s \xrightarrow{\Sigma (\eta)^s} \Sigma (\Omega \Sigma X)^s \xrightarrow{\Sigma m_s} \Sigma \Omega \Sigma X,$$ where the unlabelled arrow is the canonical splitting.
\end{enumerate}
\end{lemma}

The description of the map $\Jdec$ in Lemma \ref{GetTensor} has the following consequence. For a space $Y$, let $\mathrm{ev}: \Sigma \Omega Y \to Y$ be the evaluation map, which may be described explicitly by $\mathrm{ev}( \langle \gamma , t \rangle ) = \gamma(t)$ for $\gamma \in \Omega Y$.

\begin{lemma} \label{evaluationMap}
Let $\Jdec:\bigvee_{i=1}^\infty \Sigma X^{\wedge i} \to \Sigma \Omega \Sigma X$ be the homotopy equivalence of Lemma \ref{GetTensor}. The composite $\mathrm{ev} \circ \Jdec$ is homotopic to the projection onto the first wedge summand.
\end{lemma}

\begin{proof} Let $\iota_s : \Sigma X^{\wedge s} \to \bigvee_{i=1}^\infty \Sigma X^{\wedge i}$ be the inclusion of the $s$-th wedge summand. We must show that $$\mathrm{ev} \circ \Jdec \circ \iota_s \simeq \begin{cases} 1_{\Sigma X} & \textrm{ if } s=1 \textrm{, and}\\
* & \textrm{ otherwise.}
\end{cases}$$

The following diagram commutes up to homotopy

\begin{center}
\begin{tabular}{c}
\xymatrix{
\Sigma X \ar[r]^{\Sigma \eta} \ar[dr]_{1_{\Sigma X}} & \Sigma \Omega \Sigma X \ar[d]^{\mathrm{ev}} \\
& \Sigma X.
}
\end{tabular}
\end{center} By Lemma \ref{GetTensor}, $\Sigma \eta = \Jdec \circ \iota_1$, which implies the $s=1$ statement.

Now let $s\geq 2$. Ganea \cite[Theorems 1.1 and 1.4]{Ganea} shows that the homotopy fibre of $\mathrm{ev}$ is given by $$ \Sigma (\Omega \Sigma X \wedge \Omega \Sigma X) \xrightarrow{v} \Sigma \Omega \Sigma X \xrightarrow{\mathrm{ev}} \Sigma X,$$ where the map $v$ is equal to the composite $$ \Sigma (\Omega \Sigma X \wedge \Omega \Sigma X) \rightarrow \Sigma (\Omega \Sigma X \times \Omega \Sigma X) \xrightarrow{\Sigma m_2}  \Sigma \Omega \Sigma X$$ of $\Sigma m_2$ with the canonical splitting. We will show that $\Jdec \circ \iota_s$ factors through $v$, and hence composes trivially with $\mathrm{ev}$. Consider the following diagram, where the unlabelled arrows are all canonical splittings:

\begin{center}
\begin{tabular}{c}
\xymatrix{
& \Sigma (\Omega \Sigma X \wedge \Omega \Sigma X) \ar[r] \ar@/^1pc/[rr]^{v} & \Sigma (\Omega \Sigma X \times \Omega \Sigma X) \ar[r]_(.6){\Sigma m_2} & \Sigma \Omega \Sigma X \\
\Sigma (\Omega \Sigma X)^{\wedge s} \ar[r] & \Sigma ((\Omega \Sigma X)^{(s-1)} \wedge \Omega \Sigma X) \ar[u]^{\Sigma (m_{(s-1)} \wedge 1)} \ar[r] & \Sigma (\Omega \Sigma X)^{s}  \ar[u]^{\Sigma (m_{(s-1)} \times 1)} \ar[r]_{\Sigma m_s} & \Sigma \Omega \Sigma X \ar@{=}[u] \\
\Sigma X^{\wedge s} \ar[u]^{\Sigma \eta ^{\wedge s}} \ar[r] & \Sigma (X^{(s-1)} \wedge X) \ar[u]^{\Sigma (\eta^{(s-1)} \wedge \eta)} \ar[r] & \Sigma X^{s}. \ar[u]^{\Sigma \eta ^{s}}& 
}
\end{tabular}
\end{center}

The composite along the bottom of the diagram is $\Jdec \circ \iota_s$, so to obtain the desired factorization of $\Jdec \circ \iota_s$ through $v$, it suffices to show that the diagram commutes up to homotopy.

The top right square commutes because $m_2 \circ m_{(s-1)} \simeq m_s$, by homotopy associativity of the $H$-space $\Omega \Sigma X$. The remaining three squares commute by naturality of our canonical splitting. This completes the proof. \end{proof}

Let $\rho_k$ be the projection $T(\widetilde{K}^\mathrm{TF}_{*}(X)) \rightarrow \widetilde{K}^\mathrm{TF}_{*}(X)^{\otimes k}.$ The next corollary is immediate from Lemma \ref{evaluationMap}.

\begin{corollary} \label{desuspendedEvaluation} $S^{-1} (\mathrm{ev}_* \circ \Jdec_*) = \rho_1 : T(\widetilde{K}^\mathrm{TF}_{*}(X)) \to \widetilde{K}^\mathrm{TF}_{*}(X).$ \qed \end{corollary}

\subsection{Primitives and commutators}

It follows from the K\"unneth Theorem (Theorem \ref{KKun}), and the fact that $\Sigma (Y \times Y) \simeq \Sigma Y \vee \Sigma Y \vee \Sigma (Y \wedge Y),$ that $K^{\mathrm{TF}}_*(Y \times Y) \cong K_*^{\mathrm{TF}}(Y) \otimes K_*^{\mathrm{TF}}(Y)$. We may therefore make the following definition. A class $y \in \widetilde{K}^{\mathrm{TF}}_*(Y)$ is called \emph{primitive} if $\Delta_*(y) = y \otimes 1 + 1 \otimes y$, where $\Delta: Y \to Y \times Y$ is the diagonal, defined by $\Delta(y)=(y,y)$.

The comultiplication $Y \to Y \vee Y$ on a co-$H$-space $Y$ is a factorization of $\Delta$ via the inclusion $Y \vee Y \xhookrightarrow{} Y \times Y$. From this point of view, the following lemma is immediate.

\begin{lemma} \label{coHprim} If $Y$ is a co-$H$-space, then all elements in $\widetilde{K}^{\mathrm{TF}}_*(Y)$ are primitive. \qed
\end{lemma}

If $Y$ is an $H$-group, then the multiplication $m: Y \times Y \to Y$ induces a map $\widetilde{K}^{\mathrm{TF}}_*(Y) \otimes \widetilde{K}^{\mathrm{TF}}_*(Y) \to \widetilde{K}^{\mathrm{TF}}_*(Y)$. We will denote this map by juxtaposition, so that $m_*(y_1 \otimes y_2) = y_1 y_2$. Furthermore, the commutator $Y \times Y \to Y$ descends to a map $c: Y \wedge Y \to Y$. Expanding the definition of the commutator in terms of the $K$-homology K\"unneth Theorem (Theorem \ref{KKun}) gives the following lemma.

\begin{lemma} \label{getCommutator} Let $Y$ be an $H$-group, and let $c: Y \wedge Y \to Y$ be the commutator. If  $y_1$ and $y_2 \in \widetilde{K}^{\mathrm{TF}}_*(Y)$ are primitive, then $c_*(y_1 \otimes y_2) = y_1 y_2 - (-1)^{|y_1 | |y_2 |}y_2 y_1$. \qed
\end{lemma}

\section{The category of $\psi$-modules} \label{psiModSection}

In \cite{AdamsIV}, Adams defines an abelian category which we will follow Selick \cite{Selick} in calling $\psi$-modules. The $e$-invariant, which is our central tool, is defined by Adams in terms of $\psi$-modules. The purpose of this section is to record results about $\psi$-modules for later use.

A $\psi$-module consists of an abelian group $M$, with homomorphisms $$\psi^\ell:M \to M$$ for each $\ell \in \mathbb{Z}$, satisfying the axioms of \cite[Section 6]{AdamsIV}. If $X$ is a space then the group $\widetilde{K}^0(X)$, together with its Adams operations, is a $\psi$-module. Since we defined $\widetilde{K}^{-1}(X)$ by setting $\widetilde{K}^{-1}(X)=\widetilde{K}^{0}(\Sigma X)$, it too has the structure of a $\psi$-module. Maps of spaces induce maps of $\psi$-modules. The Adams operation $\psi^\ell$ on $\widetilde{K}^0(S^{2n})$ is multiplication by $\ell^n$, so in particular Adams operations do not commute with the Bott isomorphism.

For graded $\psi$-modules $M$ and $N$ we will write $\mathrm{Hom}_{\psi \mathrm{-Mod}}(M,N)$ for the abelian group consisting of graded $\psi$-module homomorphisms. The unadorned notation $\mathrm{Hom}(M,N)$ will mean homomorphisms of the underlying graded abelian groups.

\begin{lemma} \label{splitPsiInclusion} \begin{sloppypar} Let $M$ and $N$ be $\psi$-modules, with $N$ torsion-free. The inclusion of $\mathbb{Z}$-modules ${\mathrm{Hom}_{\psi \mathrm{-Mod}}(M,N) \xhookrightarrow{} \mathrm{Hom}(M,N)}$ is an injection onto a summand. \end{sloppypar}
\end{lemma}

\begin{proof} Let $\varphi:M \to N$ be a homomorphism of underlying $\mathbb{Z}$-modules. If, for some $k \in \mathbb{Z} \setminus \{ 0\}$, $k \cdot \varphi$ is a $\psi$-module homomorphism, then, since $N$ is torsion-free, $\varphi$ is also a $\psi$-module homomorphism. This implies that $\mathrm{Coker}(\mathrm{Hom}_{\psi \mathrm{-Mod}}(M,N) \xhookrightarrow{} \mathrm{Hom}(M,N))$ is torsion-free, which implies the result.
\end{proof}

For the avoidance of doubt, by the $e$-invariant we will always mean what Adams calls the complex $e$-invariant $e_C$ \cite{AdamsII,AdamsIV}.

\begin{definition}[Adams' $e$-invariant] \label{e} Suppose that $f:X \to Y$ induces the trivial map on $\widetilde{K}^*$. Then the cofibre sequence of $f$ gives a short exact sequence of $\psi$-modules $$0 \leftarrow \widetilde{K}^0(Y) \leftarrow \widetilde{K}^0(C_f) \leftarrow \widetilde{K}^0(\Sigma X) \leftarrow 0.$$ The \emph{$e$-invariant} of $f$ is the element of $\mathrm{Ext}_{\psi \mathrm{-Mod}}(\widetilde{K}^0(Y),\widetilde{K}^0(\Sigma X))$ represented by this exact sequence. \end{definition}

The $e$-invariant does not commute with the Bott isomorphism, but the interaction between the Bott isomorphism and the Adams operations is easy to describe, as follows. Let $\psi^\ell_Y$ be the homomorphism $\psi^\ell : \widetilde{K}^0(Y) \to \widetilde{K}^0(Y)$. Then, modulo the Bott isomorphism, we have $\psi^\ell_{\Sigma^2 X} = \ell \cdot \psi^\ell_X$. That is `upon double suspending, the Adams operations gain a factor $\ell$'. In terms of the $e$-invariant, all we need to know is the following.

\begin{lemma}{\cite[Proposition 3.4b)]{AdamsIV}} \label{eOfSuspension} There is a homomorphism $$T: \mathrm{Ext}_{\psi \mathrm{-Mod}}(\widetilde{K}^0(Y),\widetilde{K}^0(\Sigma X)) \to \mathrm{Ext}_{\psi \mathrm{-Mod}}(\widetilde{K}^0(\Sigma^2 Y),\widetilde{K}^0(\Sigma^3 X)),$$ such that $T(e(f))=e(\Sigma^2 f)$. \qed \end{lemma}

We will be concerned only with the $e$-invariants of maps whose domain is a sphere. One of the two $K$-groups of a sphere vanishes, in the dimension matching the parity of the sphere, but the $e$-invariant, as defined above, lives only in $K^0$. In order to detect maps regardless of the parity of the sphere on which they are defined, we will need to keep track of the $e$-invariants of $f$ and $\Sigma f$, so we will use the following modified $e$-invariant.

\begin{definition}[Double $e$-invariant] \label{eDouble} Let $$\mathrm{Ext}_{\psi \mathrm{-Mod}}(\widetilde{K}^*(Y),\widetilde{K}^*(\Sigma X)) := \mathrm{Ext}_{\psi \mathrm{-Mod}}(\widetilde{K}^0(Y),\widetilde{K}^0(\Sigma X)) \oplus \mathrm{Ext}_{\psi \mathrm{-Mod}}(\widetilde{K}^{-1}(Y),\widetilde{K}^{-1}(\Sigma X)).$$ Suppose that $f:X \to Y$ induces the trivial map on $\widetilde{K}^*$. Then the \emph{double $e$-invariant of f} is $\de (f)=(e(f), e(\Sigma f)) \in \mathrm{Ext}_{\psi \mathrm{-Mod}}(\widetilde{K}^*(Y),\widetilde{K}^*(\Sigma X)).$ \end{definition}

Pullback of an extension along a homomorphism defines a map $$\mathrm{Hom}_{\psi \mathrm{-Mod}}(M,B) \otimes \mathrm{Ext}_{\psi \mathrm{-Mod}}(B,A) \to \mathrm{Ext}_{\psi \mathrm{-Mod}}(M,A).$$

If $g : Y \to Z$ then $e(g \circ f)$ is represented by the pullback of $e(f)$ and $g^*: \widetilde{K}^0(Z) \to \widetilde{K}^0(Y)$ \cite[Proposition 3.2 b)]{AdamsIV}. To describe $\de ( g \circ f )$ we need only apply this result degree-wise, as follows. For convenience, we write $g^* \cdot e(f)$ for the pullback of $g^*$ and $e(f)$. Define the map $$\theta_0(f): \mathrm{Hom}_{\psi \mathrm{-Mod}}(\widetilde{K}^0(Z),\widetilde{K}^0(Y)) \to \mathrm{Ext}_{\psi \mathrm{-Mod}}(\widetilde{K}^0(Z),\widetilde{K}^0(\Sigma X))$$ $$\theta_0(f)(x) = x \cdot e(f).$$ Likewise, define $$\theta_{-1}(f) : \mathrm{Hom}_{\psi \mathrm{-Mod}}(\widetilde{K}^{-1}(Z),\widetilde{K}^{-1}(Y)) \to \mathrm{Ext}_{\psi \mathrm{-Mod}}(\widetilde{K}^{-1}(Z),\widetilde{K}^{-1}(\Sigma X))$$ $$\theta_{-1}(f)(x) = x \cdot e(\Sigma f).$$

Combining these, let $$\theta(f) : \mathrm{Hom}_{\psi \mathrm{-Mod}}(\widetilde{K}^*(Z),\widetilde{K}^*(Y)) \to \mathrm{Ext}_{\psi \mathrm{-Mod}}(\widetilde{K}^*(Z),\widetilde{K}^*(\Sigma X))$$ be the direct sum $\theta_0(f) \oplus \theta_{-1}(f)$. These definitions, together with Adams' above result, give the following lemma.

\begin{lemma} \label{eTimesD} For maps $f:X \to Y$ and $g:Y \to Z$, the following diagram commutes: \begin{center}
\begin{tabular}{c}
\xymatrix{ [Y,Z] \ar[r]^{f^*} \ar[d]_{\mathrm{deg}} & [X,Z] \ar[d]^{\de} \\
\mathrm{Hom}_{\psi \mathrm{-Mod}}(\widetilde{K}^*(Z),\widetilde{K}^*(Y)) \ar[r]^{\theta(f)} & \mathrm{Ext}_{\psi \mathrm{-Mod}}(\widetilde{K}^*(Z),\widetilde{K}^*(\Sigma X)).
}
\end{tabular}
\end{center}

\end{lemma}

Following \cite{Selick}, write $\mathbb{Z}(n)$ for the $\psi$-module $\widetilde{K}^0(S^{2n})$. Explicitly, $\mathbb{Z}(n)$ has underlying abelian group $\mathbb{Z}$, and $\psi^\ell$ acts by multiplication by $\ell^n$. It follows that $\widetilde{K}^{-1}(S^{2n+1}) := \widetilde{K}^{0}(S^{2n+2}) \cong \mathbb{Z}(n+1)$.

\begin{lemma}{\cite[Proposition 7.8, 7.9]{AdamsIV}} \label{QZstab} If $n<m$ then $\mathrm{Ext}_{\psi \mathrm{-Mod}}(\mathbb{Z}(n),\mathbb{Z}(m))$ injects into $\faktor{\mathbb{Q}}{\mathbb{Z}}$. The $e$-invariant of a map $f:S^{2m-1} \to S^{2n}$ may therefore be regarded as an element of $\faktor{\mathbb{Q}}{\mathbb{Z}}$. Furthermore, the value $e(f)$ in $\faktor{\mathbb{Q}}{\mathbb{Z}}$ satisfies $e(\Sigma^2 f) = e(f)$, so in particular, when $f$ is a map between spheres, $e(f)$ depends only on the stable homotopy class of $f$.  \qed
\end{lemma}

The following theorem is the main technical component of Selick's paper \cite{Selick}.

\begin{theorem}{\cite[Theorem 6]{Selick}} \label{Sthm} Let $f':S^{2m-1} \to S^{2n}$ be such that $p^{t-1}e(f') \neq 0$ in $\faktor{\mathbb{Q}}{\mathbb{Z}}$, for $p$ prime and some $t \in \mathbb{Z}^+$. Let $Y$ have the homotopy type of a finite $CW$-complex and let $g:S^{2n} \to Y$ be such that $\mathrm{Im}(g^*:\widetilde{K}^0(Y) \to \widetilde{K}^0(S^{2n}))$ contains $up^s\widetilde{K}^0(S^{2n})$, for $s \in \mathbb{Z}^+$ and $u$ prime to $p$. If $s<t$, and there exists some $\ell \in \mathbb{Z}^+$ for which $$\psi^\ell \otimes \mathbb{Q}: \widetilde{K}^0(Y) \otimes \mathbb{Q} \to \widetilde{K}^0(Y) \otimes \mathbb{Q}$$ does not have $\ell^m$ as an eigenvalue, then $e(g \circ f') \neq 0$. \qed
\end{theorem}

The following theorem of Gray \cite{Gray} will provide the map $f'$ for Theorem \ref{Sthm}. Specifically, this theorem provides a linearly spaced family of stems, each of which has a stable $p$-torsion class which is born on $S^3$ and detected by the $e$-invariant.

\begin{theorem}{\cite[Corollary of Theorem 6.2]{Gray}} \label{sphereInput} Let $p$ be an odd prime and let $j \in \mathbb{Z}^+$. Then there exists a class $f_{p,j} \in \pi_{2j(p-1)+2}(S^3)$ with $e(f_{p,j})=\frac{-1}{p} \in \faktor{\mathbb{Q}}{\mathbb{Z}}$. \qed
\end{theorem}

The corresponding 2-primary result is as follows. Adams \cite[Theorem 1.5 and Proposition 7.14]{AdamsIV} shows that, for $j>0$, the $(8j+3)$-rd stem contains a direct summand whose $2$-primary component has order 8, and that on this component the $e$-invariant is a surjection onto $\mathbb{Z}/4$. The sphere of origin of the classes in this component was deduced by Curtis in \cite{Curtis}.

\begin{theorem}{\cite{AdamsIV, Curtis}} \label{2sphereInput} Let $j \in \mathbb{Z}^+$. Then there exists a class $f_{2,j} \in \pi_{8j+6}(S^3)$ of order 4, with $e(f_{2,j})=\frac{-1}{2} \in \faktor{\mathbb{Q}}{\mathbb{Z}}$. \qed
\end{theorem}

\section{Main construction} \label{sectionMain}

Having assembled preliminaries in Sections \ref{KFacts} and \ref{psiModSection}, we can begin to work towards the proof of Theorem \ref{K-detection}. Our approach is as follows. From the data of Theorem \ref{K-detection}, we will construct a commutative diagram of (roughly) the following form, where $\mathscr{B}$ is a set and the other objects are $\mathbb{Z}$-modules.

\begin{center}
\begin{tabular}{c}
\xymatrix{ \mathscr{B}^k \ar[d] \ar[r] & \ar[d] \pi_*(\Omega \Sigma X) \\
I^k \ar[r] &  \mathrm{Ext}_{\psi \mathrm{-Mod}}(\widetilde{K}^*(\Omega \Sigma X),\widetilde{K}^*(S^*)).
}
\end{tabular}   
\end{center}

We will argue that \begin{itemize}
    \item The image of the top map consists of classes of order dividing $p$.
    \item The image of the left vertical map generates a submodule isomorphic to the weight $k$ component of the free graded Lie algebra over $\Zp$ on two generators.
    \item The bottom map is injective.
\end{itemize}

Together, these facts imply that there is a submodule of $\pi_*(\Omega \Sigma X) \cong \pi_{*+1}(\Sigma X)$, consisting of classes of order dividing $p$, and surjecting onto a module isomorphic to the weight $k$ component of the free graded Lie algebra over $\Zp$ on two generators. This submodule (which is necessarily a $\Zp$-vector space) must therefore have dimension at least $W_2(k)$ (Theorem \ref{gradedWitt}), which will imply that $\Sigma X$ is $p$-hyperbolic (Lemma \ref{linearimplieshype}).

The diagram will be obtained by juxtaposing three squares. Subsections \ref{sq1}, \ref{sq2}, and \ref{sq3} each construct one of these squares. In Subsection \ref{sq4} we put them together and prove Theorem \ref{K-detection}. Roughly speaking, the top map of the diagram should be thought of as first taking a family of Samelson products and then pulling them back along some suitable map $f$ coming from Gray's work (Theorem \ref{sphereInput}). The vertical maps should be thought of as passing from maps of spaces to $K$-theoretic invariants, and the bottom map (therefore) should be thought of as tracking the effect of the top map on these invariants.

Because of the need to work with a finite $CW$-complex in Selick's Theorem (Theorem \ref{Sthm}) we will restrict the right hand side of the diagram to instead refer to some finite skeleton $J_s(X)$ of the James construction.

\subsection{Samelson products and their Hurewicz images in $K$-homology} \label{sq1}

Let $R$ be a commutative ring with unit. We take a \emph{graded Lie algebra} over $R$ to be defined as in \cite{NeisendorferBook}. For a non-negatively graded $R$-module $V$, let $L(V)$ denote the \emph{free graded Lie algebra} \cite[Section 8.5]{NeisendorferBook}. Write $L^k(V)$ for the submodule of $L(V)$ generated by the brackets of length $k$ in the elements of $V$. We will call $L^k(V)$ the \emph{weight $k$ component} of $L(V)$. Note that this convention differs from Neisendorfer's - he writes $L(V)_k$ for the weight $k$ component.

\begin{definition} \label{SamelsonDef} Let $Y$ be an $H$-group, and let $c: Y \wedge Y \to Y$ be the commutator of Lemma \ref{getCommutator}. Let $\alpha \in \pi_N(Y)$, and let $\beta \in \pi_M(Y)$. The \emph{Samelson product} of $\alpha$ and $\beta$, written $\langle \alpha, \beta \rangle \in \pi_{N+M}(Y)$, is the composite $$\langle \alpha, \beta \rangle : S^{N+M} \cong S^N \wedge S^M \xrightarrow{\alpha \wedge \beta} Y \wedge Y \xrightarrow{c} Y.$$ \end{definition}

Samelson products are bilinear, graded anticommutative, and satisfy the graded Jacobi identity. They fail, however, to make $\pi_*(Y)$ into a graded Lie algebra over $\mathbb{Z}$ in Neisendorfer's sense \cite[Section 7]{Neisendorfer}. This is important because we want the natural map from the free Lie algebra to the corresponding tensor algebra to be an injection onto a summand (Lemma \ref{LieIncludes}).  One could define an appropriate notion of `free graded pseudo-Lie algebra', and proceed as follows with that in place of the set $\mathscr{B}(\pi_*(A))$, which we use in what follows, but we prefer to avoid making the extra definition.

For a graded $R$-module $V$, let $U(V)$ denote the graded set of homogeneous elements in $V$. Let $\mathscr{B}(V)$ be the set of formal iterated brackets of the elements of $U(V)$. Bracket gives a natural operation on $\mathscr{B}(V)$, which we write as $[x,y]$. Elements of $\mathscr{B}(V)$ are nonassociative words in the elements of $U(V)$, so we may define a grading on $\mathscr{B}(V)$ which extends the grading on $U(V)$ via $|[x,y]|=|x|+|y|$. The \emph{weight} of an element of $\mathscr{B}(V)$ is its word length. Write $\mathscr{B}_N(V)$ for the subset of elements in degree $N$, $\mathscr{B}^k(V)$ for the subset of elements of weight $k$, and set $\mathscr{B}_N^k(V) = \mathscr{B}^k(V) \cap \mathscr{B}_N(V)$.

Let $\nu: A \to \Omega \Sigma X$ be a map. By definition of $\mathscr{B}(\pi_*(A))$, there exists a map $$\PhiPi : \mathscr{B}(\pi_*(A)) \to \pi_*(\Omega \Sigma X)$$ which extends $\nu_*$ and satisfies $\PhiPi([x,y])=\langle \PhiPi(x), \PhiPi(y) \rangle$ for all $x,y \in \mathscr{B}(\pi_*(A))$.

For a $\mathbb{Z}/2$-graded $\mathbb{Z}$-module $V$, we define a non-negatively graded $\mathbb{Z}$-module $\mathrm{Hom}(\widetilde{K}_*^{\mathrm{TF}}(S^*),V)$, by setting $$\mathrm{Hom}(\widetilde{K}_*^{\mathrm{TF}}(S^*),V)_N = \begin{cases} \mathrm{Hom}(\widetilde{K}_*^{\mathrm{TF}}(S^N),V) & \textrm{ if } N > 0, \textrm{ and } \\
0 & \textrm{ if } N \leq 0,
\end{cases}$$ where the homomorphisms are understood to respect the $\mathbb{Z}/2$-grading on $\widetilde{K}_*$ and $V$.

In the case that $V=L$ is a $\mathbb{Z}/2$-graded Lie algebra over $\mathbb{Z}$, $\mathrm{Hom}(\widetilde{K}_*^{\mathrm{TF}}(S^*),L)$ inherits a non-negatively graded Lie algebra structure as follows. Let the generators $\xi_N$ of $\widetilde{K}_*^{\mathrm{TF}}(S^N)$ be as in Remark \ref{sphGenChoice}. Then the bracket $[f,g]$ of $f \in \mathrm{Hom}(\widetilde{K}_*^{\mathrm{TF}}(S^N),L)$ and $g \in \mathrm{Hom}(\widetilde{K}_*^{\mathrm{TF}}(S^M),L)$ is the homomorphism $\widetilde{K}_*^{\mathrm{TF}}(S^M) \to L$ carrying $\xi_{N+M}$ to $[f(\xi_N),g(\xi_M)] \in L$. The squaring operation is defined in the same way. Likewise, if $V$ is a $\mathbb{Z}/2$-graded associative algebra over $\mathbb{Z}$, then $\mathrm{Hom}(\widetilde{K}_*^{\mathrm{TF}}(S^*),V)$ inherits the structure of a non-negatively graded associative algebra.

Let $\nu :A \to \Omega \Sigma X$ be a map. There is a composition $$L(\widetilde{K}_*^{\mathrm{TF}}(A)) \to T(\widetilde{K}_*^{\mathrm{TF}}(A)) \to \widetilde{K}_*^{\mathrm{TF}}(\Omega \Sigma X),$$ where the first map is the natural map which is the identity on $\widetilde{K}_*^{\mathrm{TF}}(A)$ and satisfies $[x,y] \mapsto xy -(-1)^{|x||y|}yx$, and the second map is obtained by applying the universal property of the tensor algebra to $\nu_*$. Let $$\PhiK: \mathrm{Hom}(\widetilde{K}_*^{\mathrm{TF}}(S^*),L(\widetilde{K}_*^{\mathrm{TF}}(A))) \to \mathrm{Hom}(\widetilde{K}_*^{\mathrm{TF}}(S^*),\widetilde{K}_*^{\mathrm{TF}}(\Omega \Sigma X))$$ be the pushforward along the above composite. It is then automatic that $\PhiK$ is a map of non-negatively graded Lie algebras over $\mathbb{Z}$, where the structures are defined as above.

We write $\mathrm{deg}: \pi_N(Y) \to \mathrm{Hom}(\widetilde{K}_*^{\mathrm{TF}}(S^N),\widetilde{K}_*^{\mathrm{TF}}(Y))$ for the map $f \mapsto f_*$. Let $\mathrm{deg}': \mathscr{B}(\pi_*(A)) \to \mathrm{Hom}(\widetilde{K}_*^{\mathrm{TF}}(S^*),L(\widetilde{K}_*^{\mathrm{TF}}(A)))$ be the unique map which restricts to $\mathrm{deg}: \pi_*(A) \to \mathrm{Hom}(\widetilde{K}_*^{\mathrm{TF}}(S^*),\widetilde{K}_*^{\mathrm{TF}}(A)) \subset \mathrm{Hom}(\widetilde{K}_*^{\mathrm{TF}}(S^*),L(\widetilde{K}_*^{\mathrm{TF}}(A)))$ and carries brackets to brackets. The above maps are related as follows.

\begin{lemma} \label{hurewiczing} Let $\nu : A \to \Omega \Sigma X$, for spaces $A$ and $X$ having the homotopy type of finite $CW$-complexes. The following diagram commutes:

\begin{center}
\begin{tabular}{c}
\xymatrix{ \mathscr{B}(\pi_*(A)) \ar[r]^{\PhiPi} \ar[d]^{\mathrm{deg}'} & \pi_*(\Omega \Sigma X) \ar[d]^{\mathrm{deg}} \\
\mathrm{Hom}(\widetilde{K}_*^{\mathrm{TF}}(S^*),L(\widetilde{K}_*^{\mathrm{TF}}(A))) \ar[r]^{\PhiK} & \mathrm{Hom}(\widetilde{K}_*^{\mathrm{TF}}(S^*),\widetilde{K}_*^{\mathrm{TF}}(\Omega \Sigma X)). }
\end{tabular}   
\end{center}
\end{lemma}

\begin{proof} By construction of $\mathscr{B}(\pi_*(A))$, it suffices to show that the restriction of the diagram to the weight 1 component $\mathscr{B}^1(\pi_*(A)) = \pi_*(A)$ commutes, and that all maps respect the bracket operations. By definition, $L^1(\widetilde{K}_*^{\mathrm{TF}}(A)) = \widetilde{K}_*^{\mathrm{TF}}(A)$. It then follows immediately from the definitions of $\PhiPi$ and $\PhiK$ that restricting the left hand side of the diagram to weight 1 components gives the diagram \begin{center}
\begin{tabular}{c}
\xymatrix{ \pi_*(A) \ar[r]^{\nu_*} \ar[d]^{\mathrm{deg}} & \pi_*(\Omega \Sigma X) \ar[d]^{\mathrm{deg}} \\
\mathrm{Hom}(\widetilde{K}_*^{\mathrm{TF}}(S^*),\widetilde{K}_*^{\mathrm{TF}}(A)) \ar[r]^{\nu_*} & \mathrm{Hom}(\widetilde{K}_*^{\mathrm{TF}}(S^*),\widetilde{K}_*^{\mathrm{TF}}(\Omega \Sigma X)), }
\end{tabular}   
\end{center} which commutes, since it just expresses naturality of $\mathrm{deg}$.

It remains to show that all maps respect bracket operations. The maps $\PhiPi$ and $\mathrm{deg}'$ respect the bracket operations by definition, and $\PhiK$ respects bracket operations by construction. We therefore only need show that $\mathrm{deg}$ respects brackets. Let $f \in \pi_N(\Omega \Sigma X)$, and let $g \in \pi_M(\Omega \Sigma X)$. We must show that $\mathrm{deg}(\langle f, g \rangle)$ is the commutator $\mathrm{deg}(f)\mathrm{deg}(g)-(-1)^{NM}\mathrm{deg}(g)\mathrm{deg}(f)$ with respect to the algebra operation on $\mathrm{Hom}(\widetilde{K}_*^{\mathrm{TF}}(S^*),\widetilde{K}_*^{\mathrm{TF}}(\Omega \Sigma X))$. 

Since $\widetilde{K}_*^{\mathrm{TF}}(S^{N+M}) \cong \mathbb{Z}$, it suffices to show that the two homomorphisms agree on the generator $\xi_{N+M}$ (Remark \ref{sphGenChoice}). By Definition \ref{SamelsonDef} and the K\"unneth Theorem (Theorem \ref{KKun}), $$\mathrm{deg}(\langle f, g \rangle) (\xi_{N+M}) = c_* \circ (f_* \otimes g_*) (\xi_N \otimes \xi_M) = c_* \circ (f_*(\xi_N) \otimes g_*(\xi_M)).$$

Spheres of dimension at least 1 are co-$H$ spaces, so by Lemma \ref{coHprim}, $\xi_N$ and $\xi_M$ are primitive. By naturality of the diagonal $f_*(\xi_N)$ and $g_*(\xi_M)$ are still primitive, so by Lemma \ref{getCommutator}, $$c_* \circ (f_*(\xi_N) \otimes g_*(\xi_M)) = f_*(\xi_N)g_*(\xi_M)-(-1)^{NM}g_*(\xi_M)f_*(\xi_N),$$ which by definition of the multiplication on $\mathrm{Hom}(\widetilde{K}_*^{\mathrm{TF}}(S^*),\widetilde{K}_*^{\mathrm{TF}}(\Omega \Sigma X))$ is the result of evaluating $\mathrm{deg}(f)\mathrm{deg}(g)-(-1)^{NM}\mathrm{deg}(g)\mathrm{deg}(f)$ on $\xi_{N+M}$, as required. \end{proof}

We now lift the previous result to $J_s(X)$, thereby producing the first square of the diagram promised at the start of this section. Recall that we write $i_s : J_s(X) \to JX$ for the inclusion, and that by Theorem \ref{JamesConstruction} we have a homotopy equivalence $JX \xrightarrow{\simeq} \Omega \Sigma X$. We will abuse notation and also write $i_s$ for the composite $J_s(X) \to JX \xrightarrow{\simeq} \Omega \Sigma X$.

\begin{corollary} \label{RestrictDgrm} Let $\nu : A \to \Omega \Sigma X$, for spaces $A$ and $X$ having the homotopy type of finite $CW$-complexes, with $X$ $(r-1)$-connected for $r \geq 1$. If $N,s \in \mathbb{Z}^+$ satisfy $N < r(s+1)-1$, then $(i_s)_* : \pi_N(J_s X) \to \pi_N(\Omega \Sigma X)$ is an isomorphism and for each $k \leq s$ there exists a commutative diagram: \begin{center}
\begin{tabular}{c}
\xymatrix{ \mathscr{B}_N^k(\pi_*(A)) \ar[r]^{\widetilde{\PhiPi}} \ar[d]^{\mathrm{deg}'} & \pi_N(J_s X) \ar[d]^{\mathrm{deg}} \\
\mathrm{Hom}(\widetilde{K}_*^{\mathrm{TF}}(S^N),L^k(\widetilde{K}_*^{\mathrm{TF}}(A))) \ar[r]^{\widetilde{\PhiK}} & \mathrm{Hom}(\widetilde{K}_*^{\mathrm{TF}}(S^N),\widetilde{K}_*^{\mathrm{TF}}(J_s X)), } 
\end{tabular}   
\end{center} with $(i_s)_* \circ \widetilde{\PhiPi} = \PhiPi$ and $\mathrm{Hom}(\widetilde{K}_*^{\mathrm{TF}}(S^N),(i_s)_*) \circ \widetilde{\PhiK} = \PhiK$. \end{corollary}

\begin{proof} Consider the diagram of Lemma \ref{hurewiczing}. Lemma \ref{jTrunk} shows that $(i_s)_*$ is an isomorphism on $\pi_N$, so let $\widetilde{\PhiPi}$ be the unique map such that the condition $(i_s)_* \circ \widetilde{\PhiPi} =  \PhiPi$ holds. By Theorem \ref{JamesConstruction} (2) and Lemma \ref{GetTensor}, the map $(i_s)_*: \widetilde{K}_N^{\mathrm{TF}}(J_s(X)) \to \widetilde{K}_N^{\mathrm{TF}}(\Omega \Sigma X)$ is the inclusion of the tensors of length at most $s$. Since $k \leq s$, we may therefore define $\widetilde{\PhiK}$ to be the unique map such that the condition $\mathrm{Hom}(\widetilde{K}_*^{\mathrm{TF}}(S^N),(i_s)_*) \circ \widetilde{\PhiK} =  \PhiK$ holds. Commutativity then follows from Lemma \ref{hurewiczing} by naturality of $\mathrm{deg}$, since $\mathrm{Hom}(\widetilde{K}_*^{\mathrm{TF}}(S^N),(i_s)_*)$ is injective. \end{proof}

\begin{lemma} \label{LieIncludes} Let $V$ be a non-negatively- or $\mathbb{Z}/2$-graded $\mathbb{Z}$-module which is free and finitely generated in each dimension. Then \begin{itemize}
    \item $L(V)$ and $T(V)$ are free $\mathbb{Z}$-modules in every dimension.
    \item The natural map $L(V) \to T(V)$, $[x,y] \mapsto xy -(-1)^{|x||y|}yx$ is an injection onto a summand.
\end{itemize} \end{lemma}

\begin{proof} The non-negatively graded case is immediate from \cite{NeisendorferBook}, Proposition 8.3.1 and p282. For the $\mathbb{Z}/2$-graded case, first observe that there is a forgetful functor $U$ from $\mathbb{Z}$-graded modules to $\mathbb{Z}/2$-modules which carries $\mathbb{Z}$-graded (Lie) algebras to $\mathbb{Z}/2$-graded (Lie) algebras, and a functor $C$ from $\mathbb{Z}/2$-graded modules to $\mathbb{Z}$-modules which puts $V_0$ in any even dimension and $V_1$ in any odd dimension. Both $C$ and $U$ respect freeness and split injections, and there are natural isomorphisms $UT(CV) \cong T(V)$ and $UL(CV) \cong L(V)$. This implies the $\mathbb{Z}/2$-graded result. \end{proof}

The graded version of Theorem \ref{ungradedWitt} now follows immediately from Hilton's paper:

\begin{theorem}{\cite[Theorem 3.2, 3.3]{Hilton}} \label{gradedWitt} Let $V$ be a torsion-free $\mathbb{Z}$- or $\mathbb{Z}/2$-graded $\mathbb{Z}$-module of total dimension $n$. Then the total dimension of $L^k(V)$ is $W_n(k)$. \qed \end{theorem}

Let $R$ be a commutative ring with unit. Let $M$ be an $R$-module, and as usual let $T(M)$ denote the tensor algebra on $M$. Let $\iota_k: M^{\otimes k} \to T(M)$ be the inclusion, and let $\rho_k : T(M) \to M^{\otimes k}$ be the projection. Let $\tau: T(M) \to T(M)$ be the composite $\iota_1 \circ \rho_1$. Given an $R$-algebra $A$, and a map $\varphi : M \to A$, we write $\widetilde{\varphi}$ for the induced map $T(M) \to A$, that is, the unique map of algebras such that $\widetilde{\varphi} \circ \iota_1 = \varphi$.

Now, let $M$ and $N$ be $R$-modules, and let $\varphi : M \to T(N)$ be a map. In the proof of Theorem \ref{injectionsBeget}, we will wish to make a `leading terms' style argument. This is made precise in the next Lemma, which compares $\widetilde{\varphi}$ with $\widetilde{\tau \circ \varphi}$. Informally, we think of $\widetilde{\tau \circ \varphi}$ as the `leading terms part' of $\widetilde{\varphi}$.

\begin{lemma} \label{leadingTerms} Let $R$ be a commutative ring with unit. Let $M$ and $N$ be $\mathbb{Z}$- or $\mathbb{Z}/2$-graded $R$-modules. Let $\iota_k: M^{\otimes k} \to T(M)$ be the inclusion, let $\rho_k : T(N) \to N^{\otimes k}$ be the projection, and let $\tau : T(N) \to T(N)$ be as above. Let $\varphi:M \to T(N)$ be a map. Then $\rho_k \circ \widetilde{\varphi} \circ \iota_k = \rho_k \circ \widetilde{\tau \circ \varphi} \circ \iota_k$.
\end{lemma}

\begin{proof} It suffices to check equality on basic tensors. Let $v \in M^{\otimes k}$ be a basic tensor, so that $v=v_1 \otimes v_2 \otimes \dots \otimes v_k$, for $v_i \in M$. Then $$ \widetilde{\varphi} \circ \iota_k (v) = \widetilde{\varphi} (v_1 \otimes v_2 \otimes \dots \otimes v_k) = \varphi(v_1) \otimes \varphi(v_2) \otimes \dots \otimes \varphi(v_k)$$ $$ = \tau( \varphi(v_1)) \otimes \tau (\varphi(v_2)) \otimes \dots \otimes \tau(\varphi(v_k)) + \textrm{terms of weight $>k$}.$$

Applying $\rho_k$ to both sides yields the result. \end{proof}

\begin{theorem} \label{injectionsBeget} Let $\mathbb{F} = \mathbb{Q}$ or $\mathbb{Z}/p$ for $p$ prime. Let $\nu : A \to \Omega \Sigma X$, for spaces $A$ and $X$ having the homotopy type of finite $CW$-complexes. Let $\overline{\nu} : \Sigma A \to \Sigma X$ be the adjoint of $\nu$. If $$\overline{\nu}_* \otimes \mathbb{F} : \widetilde{K}_*^{\mathrm{TF}}(\Sigma A) \otimes \mathbb{F} \to \widetilde{K}_*^{\mathrm{TF}}(\Sigma X) \otimes \mathbb{F}$$ is an injection, then $$\PhiK \otimes \mathbb{F} : \mathrm{Hom}(\widetilde{K}_*^{\mathrm{TF}}(S^*),L(\widetilde{K}_*^{\mathrm{TF}}(A))) \otimes \mathbb{F} \to \mathrm{Hom}(\widetilde{K}_*^{\mathrm{TF}}(S^*),\widetilde{K}_*^{\mathrm{TF}}(\Omega \Sigma X)) \otimes \mathbb{F}$$ is also an injection.
\end{theorem}

\begin{remark} In the case where $\overline{\nu}$ is a suspension $\Sigma \zeta$, we have a diagram \begin{center}
\begin{tabular}{c}
\xymatrix{ \Omega \Sigma A \ar[r]^{\Omega \overline{\nu}} & \Omega \Sigma X \\
A \ar[u]^{\eta} \ar[ur]^{\nu} \ar[r]^{\zeta} & X, \ar[u]^{\eta}}
\end{tabular}
\end{center} so in particular $\nu_*$ factors through the weight 1 component $\widetilde{K}^\mathrm{TF}_*(X)$ of the tensor algebra decomposition of $\widetilde{K}^\mathrm{TF}_*(\Omega \Sigma X)$. This dramatically simplifies the proof, removing the need for Lemma \ref{leadingTerms}. In practice this is not a reasonable assumption - for example, the map  $\mu:S^3 \vee S^5 \to \Sigma \mathbb{C} P^2$ of Example \ref{CPN} (which plays the role of $\overline{\nu}$) does not desuspend. \end{remark}

\begin{proof} In this proof, for a space $Y$, we will identify the algebras $T(\widetilde{K}^\mathrm{TF}_*(Y))$ and $\widetilde{K}^\mathrm{TF}_*(\Omega \Sigma Y)$, omitting the isomorphism $S^{-1} \Jdec_*$ of Lemma \ref{GetTensor}. We defined $\PhiK$ to be the pushforward along a certain map $L(\widetilde{K}_*^{\mathrm{TF}}(A)) \to \widetilde{K}_*^{\mathrm{TF}}(\Omega \Sigma X)$. Call this map ${\PhiK}'$. It suffices to prove that ${\PhiK}' \otimes \mathbb{F}$ is an injection.

The triangle identities for the adjunction $\Sigma \dashv \Omega$ give a commutative diagram \begin{center}
\begin{tabular}{c}
\xymatrix{ \Omega \Sigma A \ar[r]^{\Omega \overline{\nu}} & \Omega \Sigma X \\
A. \ar[u]^{\eta} \ar[ur]^{\nu}
}
\end{tabular}
\end{center} 

Since ${\PhiK}'$is the unique map of Lie algebras extending $\nu$, we have a commuting diagram \begin{center}
\begin{tabular}{c}
\xymatrix{T(\widetilde{K}^\mathrm{TF}_*( A)) \cong \widetilde{K}^\mathrm{TF}_*(\Omega \Sigma A) \ar[r]^(.65){(\Omega \overline{\nu})_*} & \widetilde{K}^\mathrm{TF}_*(\Omega \Sigma X)\\
L(\widetilde{K}^\mathrm{TF}_*(A)), \ar@{^{(}->}[u] \ar[ur]_{{\PhiK}'}
}
\end{tabular}
\end{center} where we note that that by Lemma \ref{LieIncludes}, the natural map $L(\widetilde{K}^\mathrm{TF}_*(A)) \to T(\widetilde{K}^\mathrm{TF}_*(A))$ is an injection onto a summand. It therefore suffices to show that $(\Omega \overline{\nu})_* \otimes \mathbb{F}$ is an injection.


Let $\widetilde{(\nu_*)}$ denote the extension of $\nu_*$ to $T(\widetilde{K}^\mathrm{TF}_*(A))$, so that $\widetilde{(\nu_*)} = (\Omega \overline{\nu})_*$ (modulo the isomorphism $S^{-1} \Jdec_*$, as above). Since $(\rho_i \circ \widetilde{(\nu_*)} \circ \iota_k) =0$ for $i<k$, it further suffices to show that $(\rho_k \circ \widetilde{(\nu_*)} \circ \iota_k) \otimes \mathbb{F}$ is an injection for each $k$. By Lemma \ref{leadingTerms}, with $M = \widetilde{K}^\mathrm{TF}_*(A)$ and $N' = \widetilde{K}^\mathrm{TF}_*(X)$, we have that $\rho_k \circ \widetilde{(\nu_*)} \circ \iota_k  = \rho_k \circ \widetilde{(\tau \circ \nu_*)} \circ \iota_k.$

As previously, let $\mathrm{ev}: \Sigma \Omega Y \to Y$ denote the evaluation map. The following diagram commutes: \begin{center}
\begin{tabular}{c}
\xymatrix{
\Sigma A \ar[dr]_{\overline{ \nu}} \ar[r]^{\Sigma \nu} & \Sigma \Omega \Sigma X \ar[d]^{\mathrm{ev}} \\
& \Sigma X.
}
\end{tabular}
\end{center}

The hypothesis therefore implies that the composite $(\mathrm{ev}_* \circ \Sigma \nu_*) \otimes \mathbb{F}$ is an injection. Desuspending, we have that $(S^{-1}\mathrm{ev}_* \circ \nu_*) \otimes \mathbb{F}$ is an injection. By Lemma \ref{desuspendedEvaluation}, $$(\rho_1 \circ \nu_*) \otimes \mathbb{F}: \widetilde{K}^\mathrm{TF}_*(A) \otimes \mathbb{F} \to \widetilde{K}^\mathrm{TF}_*(X) \otimes \mathbb{F}$$ is an injection of $\mathbb{F}$-vector spaces. Thus, the image $(\rho_1 \circ \nu_*) (\widetilde{K}^\mathrm{TF}_*(A)) \otimes \mathbb{F} \subset \widetilde{K}^\mathrm{TF}_*(X) \otimes \mathbb{F}$ is a direct summand. Thus, the extension $\widetilde{(\tau \circ \nu_*)} \otimes \mathbb{F}$ is an injection, and $\widetilde{(\tau \circ \nu_*)} \widetilde{K}^\mathrm{TF}_*(A)^{\otimes k} \subset \widetilde{K}^\mathrm{TF}_*(X)^{\otimes k}$ for each $k$. This implies that $\rho_k \circ \widetilde{(\tau \circ \nu_*)} \circ \iota_k$ is an injection for each $k$, as required. \end{proof}

The following corollary, which lifts the injectivity back to $J_s(X)$, is immediate from Theorem \ref{injectionsBeget} and Lemma \ref{RestrictDgrm}.

\begin{corollary} \label{liftedInjectionsBeget} Let $\mathbb{F} = \mathbb{Q}$ or $\mathbb{Z}/p$ for $p$ prime. Let $\nu : A \to \Omega \Sigma X$, for spaces $A$ and $X$ having the homotopy type of finite $CW$-complexes, with $X$ $(r-1)$-connected for $r \geq 1$. Suppose that $N,s,k \in \mathbb{Z}^+$ satisfy $k \leq s$, so that the map $\widetilde{\PhiK}$ is as in Corollary \ref{RestrictDgrm}. If $$\overline{\nu}_* \otimes \mathbb{F} : \widetilde{K}_*^{\mathrm{TF}}(\Sigma A) \otimes \mathbb{F} \to \widetilde{K}_*^{\mathrm{TF}}(\Sigma X) \otimes \mathbb{F}$$ is an injection, then $$\widetilde{\PhiK} \otimes \mathbb{F} : \mathrm{Hom}(\widetilde{K}_*^{\mathrm{TF}}(S^N),L^k(\widetilde{K}_*^{\mathrm{TF}}(A))) \otimes \mathbb{F} \to \mathrm{Hom}(\widetilde{K}_*^{\mathrm{TF}}(S^N),\widetilde{K}_*^{\mathrm{TF}}(J_s X)) \otimes \mathbb{F}$$ is also an injection. \qed \end{corollary}

We have now established all that we will need to know about this `first square'.

\subsection{Maps derived from the universal coefficient isomorphism} \label{sq2}

In this subsection we will build the second square of our diagram. This square is really just the Universal Coefficient theorem (Corollary \ref{usableUCTs}) in a different form. We will write $\mathrm{deg}$ for both $K$-homological and $K$-theoretic degree.

\begin{lemma} \label{UCTmap} Let $Y$ be a space having the homotopy type of a finite $CW$-complex. There exists an isomorphism $\mathscr{U}$ making the following diagram commute.

\begin{center}
\begin{tabular}{c}
\xymatrix{
\pi_N(Y) \ar[d]^{\mathrm{deg}} \ar@{=}[r] & \pi_N(Y) \ar[d]^{\mathrm{deg}} \\
\mathrm{Hom}(\widetilde{K}_*^{\mathrm{TF}}(S^N),\widetilde{K}_*^{\mathrm{TF}}(Y)) \ar[r]^{\mathscr{U}} & \mathrm{Hom}(\widetilde{K}^*_{\mathrm{TF}}(Y),\widetilde{K}^*_{\mathrm{TF}}(S^N)).
} 
\end{tabular}   
\end{center} \end{lemma}

\begin{proof} For $\beta: \widetilde{K}_*^{\mathrm{TF}}(S^N) \to \widetilde{K}_*^{\mathrm{TF}}(Y)$, let $\mathscr{U}(\beta)$ be the unique map making the following diagram commute

\begin{center}
\begin{tabular}{c}
\xymatrix{
\widetilde{K}^*_{\mathrm{TF}}(Y) \ar[d]^{\mathscr{U}(\beta)} \ar[r]^(.4){\cong} & \mathrm{Hom}(\widetilde{K}_*^{\mathrm{TF}}(Y), \mathbb{Z}) \ar[d]^{\mathrm{Hom}(\beta, \mathbb{Z})}  \\
\widetilde{K}^*_{\mathrm{TF}}(S^N) \ar[r]^(.4){\cong}  & \mathrm{Hom}(\widetilde{K}_*^{\mathrm{TF}}(S^N), \mathbb{Z})
} 
\end{tabular}   
\end{center} where the isomorphisms are those of Corollary \ref{usableUCTs}. Since $\widetilde{K}_*^{\mathrm{TF}}(Y)$ is a finitely generated free $\mathbb{Z}$-module, $\beta \mapsto \mathrm{Hom}(\beta , \mathbb{Z})$ is an isomorphism, so $\mathscr{U}$ is also an isomorphism. Commutativity of the diagram from the statement of this lemma is by naturality of Lemma \ref{usableUCTs}. \end{proof}

\begin{corollary} \label{UCTp} Let $Y$ be a space having the homotopy type of a finite $CW$-complex. For a $\mathbb{Z}$-module $M$, let $\tau_p : M \to M \otimes \Zp$ be the natural map. There exists an injection $\mathscr{U}'$ making the following diagram commute.

\begin{center}
\begin{tabular}{c}
\xymatrix@C=5pt{
\pi_N(Y) \ar[d]^{\mathrm{deg}} \ar@/^2.2pc/[ddr] \ar@{=}[rr] & & \pi_N(Y) \ar[d]^{\mathrm{deg}} \ar@/^2.5pc/[ddr] & \\
\mathrm{Hom}(\widetilde{K}_*^{\mathrm{TF}}(S^N),\widetilde{K}_*^{\mathrm{TF}}(Y)) \ar[dd]^{\tau_p} & & \hbox to 7em{\hss$\displaystyle\mathrm{Hom}_{\psi \mathrm{-Mod}}(\widetilde{K}^*_{\mathrm{TF}}(Y),\widetilde{K}^*_{\mathrm{TF}}(S^N))$\hss} \ar[dd]^(.3){\tau_p} \\
& \mathrm{Im}(\tau_p \circ \mathrm{deg}) \ar@{_{(}->}[dl] \ar'[r]^(.7){\mathscr{U}'}[rr] & & \mathrm{Im}(\tau_p \circ \mathrm{deg}) \ar@{_{(}->}[dl] \\
\mathrm{Hom}(\widetilde{K}_*^{\mathrm{TF}}(S^N),\widetilde{K}_*^{\mathrm{TF}}(Y)) \otimes \Zp & & \hbox to 3em{\hss$\displaystyle\mathrm{Hom}_{\psi \mathrm{-Mod}}(\widetilde{K}^*_{\mathrm{TF}}(Y),\widetilde{K}^*_{\mathrm{TF}}(S^N)) \otimes \Zp .$\hss}  & }
\end{tabular}   
\end{center}
\end{corollary}

\begin{proof} By Lemma \ref{UCTmap}, we have a commutative diagram \begin{center}
\begin{tabular}{c}
\xymatrix{
\pi_N(Y) \ar[d]^{\mathrm{deg}} \ar@{=}[r] & \pi_N(Y) \ar[d]^{\mathrm{deg}} \\
\mathrm{Hom}(\widetilde{K}_*^{\mathrm{TF}}(S^N),\widetilde{K}_*^{\mathrm{TF}}(Y)) \ar[r]^{\mathscr{U}} & \mathrm{Hom}(\widetilde{K}^*_{\mathrm{TF}}(Y),\widetilde{K}^*_{\mathrm{TF}}(S^N)).
} 
\end{tabular}   
\end{center} with $\mathscr{U}$ an isomorphism, so $\mathscr{U} \otimes \Zp$ is also an isomorphism. By Lemma \ref{splitPsiInclusion}, the map $$\mathrm{Hom}_{\psi \mathrm{-Mod}}(\widetilde{K}^*_{\mathrm{TF}}(Y),\widetilde{K}^*_{\mathrm{TF}}(S^N))  \otimes \Zp \to \mathrm{Hom}(\widetilde{K}^*_{\mathrm{TF}}(Y),\widetilde{K}^*_{\mathrm{TF}}(S^N))  \otimes \Zp$$ is an injection. Maps of spaces induce maps of $\psi$-modules on $K$-theory, so the image of $\mathscr{U} \circ \mathrm{deg}$ is contained in $\mathrm{Hom}_{\psi \mathrm{-Mod}}(\widetilde{K}^*_{\mathrm{TF}}(Y),\widetilde{K}^*_{\mathrm{TF}}(S^N))$, and hence there exists a map $\mathscr{U}'$ making the following diagram commute: \begin{center}
\begin{tabular}{c}
\xymatrix{
\mathrm{Im}(\tau_p \circ \mathrm{deg}) \ar@{^{(}->}[d] \ar@{.>}[r]^(.35){\mathscr{U}'} & \mathrm{Hom}_{\psi \mathrm{-Mod}}(\widetilde{K}^*_{\mathrm{TF}}(Y),\widetilde{K}^*_{\mathrm{TF}}(S^N))  \otimes \Zp \ar@{^{(}->}[d] \\
\mathrm{Hom}(\widetilde{K}_*^{\mathrm{TF}}(S^N),\widetilde{K}_*^{\mathrm{TF}}(Y)) \otimes \Zp \ar[r]^{\mathscr{U} \otimes \Zp \ \cong}& \mathrm{Hom}(\widetilde{K}^*_{\mathrm{TF}}(Y),\widetilde{K}^*_{\mathrm{TF}}(S^N))  \otimes \Zp.
}
\end{tabular}
\end{center}

Both vertical maps are injections, so $\mathscr{U}'$ has the required properties. \end{proof}

\subsection{Pulling back along classes defined on $S^3$} \label{sq3}

Let $f \in \pi_j(S^3)$, and let $N \geq 3$. Then, for $\omega \in \pi_N(Y)$, the composite $$ S^{N+j-3} \xrightarrow{\Sigma^{N - 3} f} S^{N} \xrightarrow{\omega} Y$$ is defined. The class $\omega \circ \Sigma^{N - 3} f$ lies in $\pi_{M-1}(Y)$, where $M-1=N + j -3$.

Thus motivated, we define the map $f_\Sigma^* : \pi_*(Y) \to \pi_*(Y)$ on $\omega \in \pi_N(Y)$ by setting $f_\Sigma^* (\omega) = (\Sigma^{N-3} f)^* \omega= \omega \circ \Sigma^{N - 3} f$. In words, $f_\Sigma^*$ pulls classes back along the appropriate suspension of $f$. Strictly speaking, $f_\Sigma^*$ is only a partial map, because it is undefined on $\pi_N$ for $N \leq 2$, but this will be unimportant.

Recall the definition of the double $e$-invariant $\de$ (Definition \ref{eDouble}). On $\pi_N(Y)$, we have by definition that $f_\Sigma^* = (\Sigma^{N-3} f)^*$. By Lemma \ref{eTimesD} we have a commuting square: \begin{center}
\begin{tabular}{c}
\xymatrix{
\pi_N(Y) \ar[rr]^{f_\Sigma^*} \ar[d]_{\mathrm{deg}} & & \pi_{N+j-3}(Y) \ar[d]^{\de} \\
\mathrm{Hom}_{\psi \mathrm{-Mod}}(\widetilde{K}^*_{\mathrm{TF}}(Y),\widetilde{K}^*_{\mathrm{TF}}(S^N)) \ar[rr]^{\theta(\Sigma^{N-3} f)} & & \mathrm{Ext}_{\psi \mathrm{-Mod}}(\widetilde{K}^*_{\mathrm{TF}}(Y),\widetilde{K}^*_{\mathrm{TF}}(S^{N+j-2})).
}
\end{tabular}
\end{center}

Mimicking the convention for $f_\Sigma^*$, let $$\theta_\Sigma(f) : \mathrm{Hom}_{\psi \mathrm{-Mod}}(\widetilde{K}^*_{\mathrm{TF}}(Y),\widetilde{K}^*_{\mathrm{TF}}(S^*)) \to \mathrm{Ext}_{\psi \mathrm{-Mod}}(\widetilde{K}^*_{\mathrm{TF}}(Y),\widetilde{K}^*_{\mathrm{TF}}(S^{*+j-2}))$$ be the map which is defined to be equal to $\theta(\Sigma^{N-3} f)$ on the degree $N$ component $\mathrm{Hom}_{\psi \mathrm{-Mod}}(\widetilde{K}^*_{\mathrm{TF}}(Y),\widetilde{K}^*_{\mathrm{TF}}(S^N))$ of $\mathrm{Hom}_{\psi \mathrm{-Mod}}(\widetilde{K}^*_{\mathrm{TF}}(Y),\widetilde{K}^*_{\mathrm{TF}}(S^*))$.

\begin{lemma} \label{lastStageWellDef} Let $p$ be a prime, and let $f \in \pi_j(S^3)$ with $e (f)$ defined. If $pf=0$, then there exists a map $\theta_\Sigma^p(f)$ making the following diagram commute for all $N$: \begin{center}
\begin{tabular}{c}
\xymatrix{
\pi_N(Y) \ar[d]_{\mathrm{deg}} \ar[r]^{f_\Sigma^*} & \pi_{N+j-3}(Y) \ar[d]^{\de} \\
\mathrm{Hom}_{\psi \mathrm{-Mod}}(\widetilde{K}^*_{\mathrm{TF}}(Y),\widetilde{K}^*_{\mathrm{TF}}(S^N)) \ar[d]_{\tau_p} \ar[r]^{\theta_\Sigma(f)} & \mathrm{Ext}_{\psi \mathrm{-Mod}}(\widetilde{K}^*_{\mathrm{TF}}(Y),\widetilde{K}^*_{\mathrm{TF}}(S^{N+j-2})) \\
\mathrm{Hom}_{\psi \mathrm{-Mod}}(\widetilde{K}^*_{\mathrm{TF}}(Y),\widetilde{K}^*_{\mathrm{TF}}(S^N))  \otimes \Zp. \ar@{.>}[ur]^{\theta_\Sigma^p(f)} &
}
\end{tabular}
\end{center}
\end{lemma}

\begin{proof} Since $pf=0$, we have that $p\de(\Sigma^{N-3}f)=0$ for all $N$, which implies that $\theta_\Sigma(f)$ vanishes on $p$-divisible elements, so there exists a unique map $\theta_\Sigma^p(f)$ making the diagram commute, as required.
\end{proof}

\begin{lemma} \label{eigenvalues} Let $X$ be a finite $CW$-complex. Let $\lambda_\ell^X$ be the largest eigenvalue of the rational Adams operation $$\psi^\ell \otimes \mathbb{Q} : \widetilde{K}^0(X) \otimes \mathbb{Q} \to \widetilde{K}^0(X) \otimes \mathbb{Q}.$$ Then, for $i \geq 0$ \begin{itemize}
    \item the largest eigenvalue of $\psi^\ell \otimes \mathbb{Q}$ on $\widetilde{K}^0(\Sigma^{2i} J_s(X)) \otimes \mathbb{Q}$ is $\ell^i(\lambda_\ell^X)^s$, and
    \item the largest eigenvalue of $\psi^\ell \otimes \mathbb{Q}$ on $\widetilde{K}^0(\Sigma^{2i+1} J_s(X)) \otimes \mathbb{Q}$ is $\ell^i\lambda_\ell^{\Sigma X}(\lambda_\ell^X)^{s-1}$.
\end{itemize}
\end{lemma}

\begin{proof} When $i \geq 1$, Theorem \ref{JamesConstruction} gives that $\Sigma J_s(X) \simeq \Sigma \bigvee_{t=1}^s X^{\wedge t}$, so $\Sigma^{2i} J_s(X) \simeq S^{2i} \wedge \bigvee_{t=1}^s X^{\wedge t}$, and $\Sigma^{2i+1} J_s(X) \simeq S^{2i} \wedge \Sigma X \wedge \bigvee_{t=1}^{s-1} X^{\wedge t}$. By the K\"unneth theorem (Theorem \ref{KUpKun}), this implies isomorphisms of rings $$\widetilde{K}^0_{\mathrm{TF}}(\Sigma^{2i} J_s(X)) \cong \bigoplus_{t=1}^s \widetilde{K}^0_{\mathrm{TF}}(S^{2i}) \otimes \widetilde{K}^0_{\mathrm{TF}}(X)^{\otimes t}, \mathrm{ for } i \geq 1,$$ and $$\widetilde{K}^0_{\mathrm{TF}}(\Sigma^{2i+1} J_s(X)) \cong \bigoplus_{t=0}^{s-1} \widetilde{K}^0_{\mathrm{TF}}(S^{2i}) \otimes \widetilde{K}^0_{\mathrm{TF}}(\Sigma X) \otimes  \widetilde{K}^0_{\mathrm{TF}}(X)^{\otimes t} \mathrm{ for } i \geq 0.$$

The K\"unneth isomorphism of Theorem \ref{KUpKun} is given by the external product of $K$-theory. Since the Adams operations are ring homomorphisms, the above isomorphisms are also isomorphisms of $\psi$-modules. In particular, the Adams operations on the left are the tensor product of the corresponding operations on the right.

These decompositions hold for $\widetilde{K}^0_{\mathrm{TF}}$, so they also hold for $\mathbb{Q} \otimes \widetilde{K}^0$, and the remaining problem is to determine the largest eigenvalue of the relevant tensor products of Adams operations. The eigenvalues of a tensor product of linear endomorphisms are precisely the products of the eigenvalues. The operation $\psi^\ell$ acts on $S^{2i}$ by multiplication by $\ell^i$. Together, these observations imply the result. \end{proof}

\begin{lemma} \label{lastStageInjection} Let $p$ be an odd prime. Let $X$ be an $(r-1)$-connected finite $CW$-complex. Let $N, s \in \mathbb{Z}^+$. Consider the diagram of Lemma \ref{lastStageWellDef} for $Y=J_s(X)$ and $f=f_{p,j} \in \pi_{2j(p-1)+2}(S^3)$, the map of Theorem \ref{sphereInput}:  \begin{center}
\begin{tabular}{c}
\xymatrix{
\pi_N(J_s(X)) \ar[d]_{\mathrm{deg}} \ar[r]^{f_\Sigma^*} & \pi_{N+2j(p-1)-1}(J_s(X)) \ar[d]^{\de} \\
\mathrm{Hom}_{\psi \mathrm{-Mod}}(\widetilde{K}^*_{\mathrm{TF}}(J_s(X)),\widetilde{K}^*_{\mathrm{TF}}(S^N)) \ar[d]_{\tau_p} \ar[r]^{\theta_\Sigma(f)} & \mathrm{Ext}_{\psi \mathrm{-Mod}}(\widetilde{K}^*_{\mathrm{TF}}(J_s(X)),\widetilde{K}^*_{\mathrm{TF}}(S^{N+2j(p-1)})) \\
\mathrm{Hom}_{\psi \mathrm{-Mod}}(\widetilde{K}^*_{\mathrm{TF}}(J_s(X)),\widetilde{K}^*_{\mathrm{TF}}(S^N))  \otimes \Zp. \ar[ur]^{\theta_\Sigma^p(f)} &
}
\end{tabular}
\end{center} For $\ell \in \mathbb{Z}^+$, let $\lambda_\ell^Y$ be the largest eigenvalue of $\psi^\ell \otimes \mathbb{Q}$ on $\widetilde{K}^0(Y) \otimes \mathbb{Q}$, and let $\lambda_\ell = \mathrm{max}(\lambda_\ell^X, \lambda_\ell^{\Sigma X})$. If there exists $\ell \in \mathbb{Z}^+$ such that $\ell^{j(p-1)+\frac{N-1}{2}}>\lambda_\ell^{s}$ then $\mathrm{Ker}(\de \circ f^*_\Sigma) \subset \mathrm{Ker}(\tau_p \circ \mathrm{deg})$, and hence the restriction of $\theta_\Sigma^p(f)$ to $\mathrm{Im}(\tau_p \circ \mathrm{deg})$ is an injection.
\end{lemma}

\begin{proof} First note that $pf=0$ by Theorem \ref{sphereInput}, so $\theta_\Sigma^p(f)$ is well-defined by Lemma \ref{lastStageWellDef}. Let $\omega \in \pi_N(J_s(X))$. Suppose that $\omega \in \mathrm{Ker}(\de \circ f^*_\Sigma)$, that is, that the $\de$-invariant of the composite $$S^{N+2j(p-1)-1} \xrightarrow{ \Sigma^{N-3} f} S^{N} \xrightarrow{\omega} J_s(X)$$ is trivial. By Lemma \ref{eOfSuspension}, this implies that $\Sigma^i(\omega \circ \Sigma^{N-3} f)$ has trivial $e$-invariant for all $i$. In particular, $e(\Sigma \omega \circ \Sigma^{N-2} f)$ and $e(\omega \circ \Sigma^{N-3} f)$ are both 0.

By Lemma \ref{eigenvalues}, the largest eigenvalue of $\psi^\ell \otimes \mathbb{Q}$ on $\widetilde{K}^0(J_s(X)) \otimes \mathbb{Q}$ is at most $\lambda_\ell^s$, and the largest eigenvalue of $\psi^\ell \otimes \mathbb{Q}$ on $\widetilde{K}^0(\Sigma J_s(X)) \otimes \mathbb{Q}$ is also at most $\lambda_\ell^s$. We now divide into cases, based on the parity of $N$.

CASE 1 ($N$ even): Write $N=2n$. Let $f'=\Sigma^{N-3} f$ and $g=\omega$ in Theorem \ref{Sthm}. The domain of $ \omega \circ \Sigma^{N-3} f$ is $S^{M-1}$, where $M-1=N + 2j(p-1) -1$, so $M$ is even, as is required. To check the eigenvalue hypothesis of Theorem \ref{Sthm}, write $M=2m$. By Lemma \ref{eigenvalues}, the largest eigenvalue of $\psi^\ell \otimes \mathbb{Q}$ on $\widetilde{K}^0(J_s(X)) \otimes \mathbb{Q}$ is at most $\lambda_\ell^s$, and $\ell^m = \ell^{j(p-1)+n} > \ell^{j(p-1)+\frac{N-1}{2}}$, which we assumed was greater than $\lambda^s_\ell$. This means that $\ell^m$ cannot be an eigenvalue of $\psi^\ell \otimes \mathbb{Q}$ on $\widetilde{K}^0(J_s(X)) \otimes \mathbb{Q}$. Now, $e(f) \neq 0$ by construction (Theorem \ref{sphereInput}), so $e(\Sigma^{N-3} f) \neq 0$ by stability (Lemma \ref{QZstab}). Since $e(\omega \circ \Sigma^{N-3} f)=0$, the contrapositive of Theorem \ref{Sthm} gives that $\omega^*$ has $p$-divisible image in $\widetilde{K}^0(S^N)$. Since $N$ is even, this implies that $\tau_p \circ \mathrm{deg}(\omega)=0$, as required.

CASE 2 ($N$ odd): Write $n=2n+1$. Let $f'=\Sigma^{N-2} f$ and $g=\Sigma \omega$ in Theorem \ref{Sthm}, and proceed similarly to case 1. The domain of $ \Sigma \omega \circ \Sigma^{N-2} f$ is $S^{M-1}$, where $M-1=N + 2j(p-1)$, so $M$ is even, as is required. To check the eigenvalue hypothesis of Theorem \ref{Sthm}, write $M=2m$. By Lemma \ref{eigenvalues}, the largest eigenvalue of $\psi^\ell \otimes \mathbb{Q}$ on $\widetilde{K}^0(\Sigma J_s(X)) \otimes \mathbb{Q}$ is at most $\lambda_\ell^s$, and $\ell^m = \ell^{j(p-1)+n} = \ell^{j(p-1)+\frac{N-1}{2}}$, which we assumed was greater than $\lambda^s_\ell$. This means that $\ell^m$ cannot be an eigenvalue of $\psi^\ell \otimes \mathbb{Q}$ on $\widetilde{K}^0(\Sigma J_s(X)) \otimes \mathbb{Q}$. As in the previous case, $e(\Sigma^{N-2} f) \neq 0$. Since $e(\Sigma \omega \circ \Sigma^{N-2} f)=0$, the contrapositive of Theorem \ref{Sthm} gives that $(\Sigma \omega)^*$ has $p$-divisible image in $\widetilde{K}^0(S^{N+1})$. Since $N$ is odd, this implies that $\tau_p \circ \mathrm{deg}(\omega)=0$, as required. This completes the case, and hence the proof. \end{proof}

\subsection{Proof of Theorem \ref{K-detection}} \label{sq4}

\begin{construction} \label{dgrmConstruction} 

Let $p$ be an odd prime. Let $\nu : A \to \Omega \Sigma X$, for spaces $A$ and $X$ having the homotopy type of finite $CW$-complexes, with $X$ $(r-1)$-connected for $r \geq 1$. Let $f \in \pi_i(S^3)$ with $\overline{e}(f)$ defined. Suppose that $N,k,s \in \mathbb{Z}^+$ satisfy $N < r(s+1)-1$ and $k \leq s$. The diagrams of the preceding subsections may be combined as follows.

Recall the definition of $\mathrm{deg}'$ from the preamble to Lemma \ref{hurewiczing}. Let $I(A)$ be the submodule of $\mathrm{Hom}(\widetilde{K}_*^{\mathrm{TF}}(S^*),L(\widetilde{K}_*^{\mathrm{TF}}(A))) \otimes \Zp$ generated by $\mathrm{Im}(\tau_p \circ \mathrm{deg}')$. The same grading conventions as usual apply: we write $I^k(A)$ for the weight $k$ part, we write $I_N(A)$ for the degree $N$ part, and let $I^k_N(A)=I^k(A) \cap I_N(A)$.

From Corollary \ref{RestrictDgrm}, using the assumptions that $N < r(s+1)-1$ and $k \leq s$ (which make $\widetilde{\PhiPi}$ and $\widetilde{\PhiK}$ well-defined) we obtain the following diagram, where the images of the vertical maps have been `popped out' to their right.

\begin{center}
\begin{tabular}{c}
\begin{tikzcd}[column sep=0em]
\mathscr{B}_N^k(\pi_*(A)) \ar[rr, "\widetilde{\PhiPi}"] \ar[dr] \ar[dd,"\tau_p \circ \mathrm{deg}'"] & & \pi_N(J_s( X)) \ar[dd,"\tau_p \circ \mathrm{deg}", near start] \ar[dr, end anchor={[yshift=1ex]}] & \\
& I_N^k(A) \ar[rr, crossing over, end anchor={[xshift=-5ex]}] \ar[dl] & & \mathclap{\mathrm{Im}(\tau_p \circ \mathrm{deg})} \ar[dl, start anchor={[yshift=-1ex]}] \\
\mathrm{Hom}(\widetilde{K}_*^{\mathrm{TF}}(S^N),L^k(\widetilde{K}_*^{\mathrm{TF}}(A))) \otimes \Zp \ar[rr,"\widetilde{\PhiK} \otimes \Zp"] & & \mathrm{Hom}(\widetilde{K}_*^{\mathrm{TF}}(S^N),\widetilde{K}_*^{\mathrm{TF}}(J_s (X))) \otimes \Zp. &
\end{tikzcd}
\end{tabular}
\end{center}

Next, from Corollary \ref{UCTp} (with $Y=J_s(X)$) we have a diagram \begin{center}
\begin{tabular}{c}
\begin{tikzcd}[column sep=0em]
\pi_N(J_s(X))  \ar[dr, end anchor={[yshift=0.5ex]}] \ar[dd,"\tau_p \circ \mathrm{deg}"] \ar[rr, equal] & & \pi_N(J_s(X)) \ar[dr, end anchor={[yshift=1ex]}] \ar[dd, "\tau_p \circ \mathrm{deg}", near start] & \\
 & \mathclap{\mathrm{Im}(\tau_p \circ \mathrm{deg})} \phantom{X} \ar[dl, start anchor={[yshift=-0.5ex]}] \ar[rr, "\mathscr{U}'", near start, crossing over, start anchor={[xshift=3ex]}, end anchor={[xshift=-5ex]}] &  & \mathclap{\mathrm{Im}(\tau_p \circ \mathrm{deg})} \ar[dl, start anchor={[yshift=-1ex]}] \\
\mathrm{Hom}(\widetilde{K}_*^{\mathrm{TF}}(S^N),\widetilde{K}_*^{\mathrm{TF}}(J_s (X))) \otimes \Zp &  & \mathrm{Hom}_{\psi \mathrm{-Mod}}(\widetilde{K}^*_{\mathrm{TF}}(J_s(X)),\widetilde{K}^*_{\mathrm{TF}}(S^N))  \otimes \Zp.
\end{tikzcd}
\end{tabular}
\end{center}

Lastly, we obtain the following diagram from Lemma \ref{lastStageWellDef}: \begin{center}
\begin{tabular}{c}
\xymatrix{
\pi_N(J_s(X)) \ar[d]_{\mathrm{deg}} \ar[r]^{f_\Sigma^*} & \pi_{N+i-3}(J_s(X)) \ar[d]^{\de} \\
\mathrm{Hom}_{\psi \mathrm{-Mod}}(\widetilde{K}^*_{\mathrm{TF}}(J_s(X)),\widetilde{K}^*_{\mathrm{TF}}(S^N)) \ar[d]_{\tau_p} \ar[r]^{\theta_\Sigma(f)} & \mathrm{Ext}_{\psi \mathrm{-Mod}}(\widetilde{K}^*_{\mathrm{TF}}(J_s(X)),\widetilde{K}^*_{\mathrm{TF}}(S^{N+i-2})) \\
\mathrm{Hom}_{\psi \mathrm{-Mod}}(\widetilde{K}^*_{\mathrm{TF}}(J_s(X)),\widetilde{K}^*_{\mathrm{TF}}(S^N))  \otimes \Zp. \ar[ur]^{\theta_\Sigma^p(f)} &
}
\end{tabular}
\end{center}

Concatenating these diagrams gives a diagram as follows: \begin{center}
\begin{tabular}{c}
\xymatrix{ \mathscr{B}_N^k(\pi_*(A)) \ar[rrr]^{f^*_\Sigma \circ \widetilde{\PhiPi}} \ar[d]_{\tau_p \circ \mathrm{deg}'} & & & \pi_{N+i-3}(J_{s}(X)) \ar[d]^{\de} \\
I_N^k(A) \ar[rrr]^(.32){\theta_\Sigma^p(f) \circ \mathscr{U}' \circ (\widetilde{\PhiK} \otimes \Zp)} & & & \mathrm{Ext}_{\psi \mathrm{-Mod}}(\widetilde{K}^*_{\mathrm{TF}}(J_{s}(X)),\widetilde{K}^*_{\mathrm{TF}}(S^{N+i-2})).
}
\end{tabular}
\end{center}
\end{construction}

In this subsection, we combine the results of the previous subsections to produce results about this diagram.

\begin{theorem} \label{Ndgrm} Let $p$ be an odd prime. Let $\nu : A \to \Omega \Sigma X$, for spaces $A$ and $X$ having the homotopy type of finite $CW$-complexes, with $X$ $(r-1)$-connected for $r \geq 1$. Let $N,k,s \in \mathbb{Z}^+$ with $N < r(s+1)-1$ and $k \leq s$. Let $f=f_{p,j} \in \pi_{2j(p-1)+2}(S^3)$, the map of Theorem \ref{sphereInput}. 

For $\ell \in \mathbb{Z}^+$, let $\lambda_\ell^Y$ be the largest eigenvalue of $\psi^\ell \otimes \mathbb{Q}$ on $\widetilde{K}^0(Y) \otimes \mathbb{Q}$, and let $\lambda_\ell = \mathrm{max}(\lambda_\ell^X, \lambda_\ell^{\Sigma X})$. If \begin{itemize}
    \item $\overline{\nu}_* \otimes \Zp : \widetilde{K}^\mathrm{TF}_*(\Sigma A) \otimes \Zp \to \widetilde{K}^\mathrm{TF}_*(\Sigma X) \otimes \Zp$ is an injection, and
    \item there exists $\ell \in \mathbb{Z}^+$ such that $\ell^{j(p-1)+\frac{N-1}{2}}>\lambda_\ell^{s}$,
\end{itemize} then $\theta_\Sigma^p(f) \circ \mathscr{U}' \circ (\widetilde{\PhiK} \otimes \Zp): I_N^k(A) \to \mathrm{Ext}_{\psi \mathrm{-Mod}}(\widetilde{K}^*_{\mathrm{TF}}(J_{s}(X)),\widetilde{K}^*_{\mathrm{TF}}(S^{N+2j(p-1)}))$ is an injection. \end{theorem}

\begin{proof} By Corollary \ref{liftedInjectionsBeget}, since $\overline{\nu}_* \otimes \Zp$ is an injection, $\widetilde{\PhiK} \otimes \Zp$ is also an injection. By Corollary \ref{UCTp} $\mathscr{U}'$ is an injection. By Lemma \ref{lastStageInjection} the hypothesis on $\ell$ implies that the restriction of $\theta_\Sigma^p(f)$ to $\mathrm{Im}(\tau_p \circ \mathrm{deg})$ is an injection. The map $\theta_\Sigma^p(f) \circ \mathscr{U}' \circ (\widetilde{\PhiK} \otimes \Zp)$  is thus a composite of injections, hence an injection, as required.\end{proof}

In the proof of Theorem \ref{K-detection}, we will wish to restrict attention to those elements of $\mathscr{B}(\pi_*(A))$ who are brackets of classes in $\pi_*(A)$ in some dimensional range $\qmin \leq n \leq \qmax$. All such classes lie in dimensions $k \qmin \leq N \leq k \qmax$. Said more precisely, we have an inclusion $\mathscr{B}^k(\bigoplus_{n=\qmin}^{\qmax} \pi_n(A)) \subset \bigcup_{N = k \qmin}^{k \qmax} \mathscr{B}^k_N(\pi_*(A))$. We will now study the diagram of Construction \ref{dgrmConstruction} in this dimensional range.

\begin{construction} \label{kConstruction} Let $p$ be an odd prime, $\nu: A \to \Omega \Sigma X$ for finite $CW$-complexes $A$ and $X$ with $X$ $(r-1)$-connected for $r \geq 1$, and $f \in \pi_i(S^3)$ with $\overline{e}(f)$ defined. Let $\qmax>\qmin$ be natural numbers. Fix $k \in \mathbb{Z}^+$, and let $s=k \qmax + 1$. For $N \in \mathbb{Z}^+$ with $k \qmin \leq N \leq k \qmax$, we have that $N < r(s + 1) -1$ and $k \leq s$. Combining the diagrams obtained from Construction \ref{dgrmConstruction} for this range of values of $N$ gives the following diagram: \begin{center}
\begin{tabular}{c}
\xymatrix{ \bigcup_{N = k\qmin}^{k \qmax} \mathscr{B}_N^k(\pi_*(A)) \ar[rrr]^{f^*_\Sigma \circ \widetilde{\PhiPi}} \ar[d]_{ \tau_p \circ \mathrm{deg}'} & & & \bigoplus_{N = k \qmin}^{k \qmax} \pi_{N+i-3}(J_s(X)) \ar[d]^{\de} \\
\bigoplus_{N = k\qmin}^{k \qmax} I^k_N(A) \ar[rrr]^(.34){\theta_\Sigma^p(f) \circ \mathscr{U}' \circ (\widetilde{\PhiK} \otimes \Zp)} & & & \bigoplus_{N = k \qmin}^{k \qmax} \mathrm{Ext}_{\psi \mathrm{-Mod}}(\widetilde{K}^*_{\mathrm{TF}}(J_s(X)),\widetilde{K}^*_{\mathrm{TF}}(S^{N+i-2})).
}
\end{tabular}
\end{center} \end{construction}

We now show that by choosing a large enough $c \in \mathbb{Z}^+$, and setting $f=f_{p,ck}$, the eigenvalue hypothesis of Theorem \ref{Ndgrm} may be satisfied across the dimensional range of Construction \ref{kConstruction} for all sufficiently large $k$.

\begin{corollary} \label{kDgrm} Let $p$ be an odd prime. Let $\nu : A \to \Omega \Sigma X$, for spaces $A$ and $X$ having the homotopy type of finite $CW$-complexes, with $X$ path-connected. Let $\qmax>\qmin$ be natural numbers. Let $c,k \in \mathbb{Z}^+$. Let $f=f_{p,ck} \in \pi_{2ck(p-1)+2}(S^3)$ be the map of Theorem \ref{sphereInput}. If $$\overline{\nu}_* \otimes \Zp : \widetilde{K}^\mathrm{TF}_*(\Sigma A) \otimes \Zp \to \widetilde{K}^\mathrm{TF}_*(\Sigma X) \otimes \Zp$$ is an injection then there exists $c \in \mathbb{Z}^+$ such that for large enough $k \in \mathbb{Z}^+$, $$\theta_\Sigma^p(f) \circ \mathscr{U}' \circ (\widetilde{\PhiK} \otimes \Zp): \bigoplus_{N = k \qmin}^{k \qmax} I_N^k(A) \to \bigoplus_{N = k \qmin}^{k \qmax} \mathrm{Ext}_{\psi \mathrm{-Mod}}(\widetilde{K}^*_{\mathrm{TF}}(J_s(X)),\widetilde{K}^*_{\mathrm{TF}}(S^{N+2ck(p-1)}))$$ is an injection.
\end{corollary}

\begin{proof} By Theorem \ref{Ndgrm}, it suffices to show that for each $N$ with $k \qmin \leq N \leq k \qmax$ there exists $\ell \in \mathbb{Z}^+$ such that $\ell^{ck(p-1)+\frac{N-1}{2}}>\lambda_\ell^s=\lambda_\ell^{k \qmax + 1}$. Take any $\ell \geq  2$. Since $N \geq k \qmin$, it suffices to find $c$ such that for large enough $k$ we have $\ell^{ck(p-1)+\frac{k \qmin -1}{2}}>\lambda_\ell^{k \qmax}$. Taking $\log$s on both sides, this is equivalent to $$(ck(p-1)+\frac{k \qmin -1}{2}) \log (\ell)>k \qmax \log(\lambda_\ell).$$ It is now clear that we may choose $c$ large enough that this equation holds for large enough $k$, in particular, any $c \geq \frac{1}{p-1}(\qmax \frac{\log(\lambda_\ell)}{\log(\ell)} - \frac{\qmin}{2})$ will do. \end{proof}

Before proving Theorem \ref{K-detection}, we prove three lemmas. The first converts the $K$-theoretic hypothesis of Theorem \ref{K-detection} into the $K$-homological input our construction requires.

\begin{lemma} \label{useOfUCT} Let $X$ be a space, and let $p$ be prime. Let $\mu : S^{q_1+1} \vee S^{q_2+1} \to \Sigma X$ be a map with $q_i \geq 1$, such that the map $$\widetilde{K}^*_{\mathrm{TF}}(\Sigma X) \otimes \Zp \xrightarrow{\mu^* \otimes \Zp} \widetilde{K}^*_{\mathrm{TF}}(S^{q_1 +1} \vee S^{q_2 +1}) \otimes \Zp \cong \Zp \oplus \Zp$$
is a surjection. Then $\mu_* \otimes \Zp : \widetilde{K}_*^{\mathrm{TF}}(S^{q_1 + 1} \vee S^{q_2 + 1}) \otimes \Zp \to \widetilde{K}_*^{\mathrm{TF}}( \Sigma X) \otimes \Zp$ is an injection. \end{lemma}

\begin{proof} Naturality of the Universal Coefficient Theorem (Corollary \ref{usableUCTs}) relative to the map $\mu$ gives a diagram of short exact sequences \begin{center}
\begin{tabular}{c}
\xymatrix{ 0 \ar[r] & 0 \ar[r] & \widetilde{K}^*_{\mathrm{TF}}(S^{q_1 + 1} \vee S^{q_2 + 1}) \ar[r] & \mathrm{Hom}(\widetilde{K}_*^{\mathrm{TF}}(S^{q_1 + 1} \vee S^{q_2 + 1}), \mathbb{Z}) \ar[r] & 0 \\
0 \ar[r] & E \ar[u] \ar[r] & \widetilde{K}^*_{\mathrm{TF}}(\Sigma X) \ar^{\mu^*}[u] \ar[r] & \mathrm{Hom}(\widetilde{K}_*^{\mathrm{TF}}(\Sigma X), \mathbb{Z}) \ar^{\mathrm{Hom}(\mu_*, \mathbb{Z})}[u] \ar[r] & 0,
}  
\end{tabular}
\end{center} where $E$ denotes the quotient of $\mathrm{Ext}(\widetilde{K}_{*-1}(\Sigma X), \mathbb{Z})$ by its torsion subgroup. We will argue by contrapositive. Suppose that $\mu_* \otimes \Zp$ is not injective. The $\Zp$-vector space $\widetilde{K}_*^{\mathrm{TF}}(S^{q_1 + 1} \vee S^{q_2 + 1}) \otimes \Zp$ has dimension 2, so there exists a basis $x,y$ where $(\mu_* \otimes \Zp) (y)=0$. Prior to tensoring with $\Zp$, this means that there exists a non-$p$-divisible element $\widetilde{y}$ of $\widetilde{K}_*^{\mathrm{TF}}(S^{q_1 + 1} \vee S^{q_2 + 1})$ such that $p$ divides $\mu_*(\widetilde{y})$. This implies that any element $\varphi$ of $\mathrm{Hom}(\widetilde{K}_*^{\mathrm{TF}}(S^{q_1 + 1} \vee S^{q_2 + 1}), \mathbb{Z})$ with $\varphi(\widetilde{y})$ not $p$-divisible is not contained in the image of $\mathrm{Hom}(\mu_*, \mathbb{Z})$, hence, by the diagram, that $\mu^* \otimes \Zp$ is not surjective, as required. \end{proof}

Some preamble to the second lemma is necessary. Let $h : \pi_*(A) \to \widetilde{K}_*^{\mathrm{TF}}(A)$ be the $K$-homological Hurewicz map, which sends $f \in \pi_N(A)$ to $f_*(\xi_N) \in \widetilde{K}_*^{\mathrm{TF}}(A)$. As with $\mathrm{deg}'$, let $h' : \mathscr{B}(\pi_*(A)) \to L(\widetilde{K}_*^{\mathrm{TF}}(A))$ be the unique map which restricts to $h : \pi_*(A) \to \widetilde{K}_*^{\mathrm{TF}}(A) \subset L(\widetilde{K}_*^{\mathrm{TF}}(A))$ and respects brackets.

Let $M$ be a $\mathbb{Z}/2$-graded $\mathbb{Z}$-module. Let $\chi : \mathrm{Hom}(\widetilde{K}_*^{\mathrm{TF}}(S^*),M) \to M$ be the map which carries $\varphi \in  \mathrm{Hom}(\widetilde{K}_*^{\mathrm{TF}}(S^N),M)$ to $\varphi(\xi_N) \in M$ (Remark \ref{sphGenChoice}). If $M=L$ is a $\mathbb{Z}/2$-graded Lie algebra, then it follows immediately from the definition of the bracket in $\mathrm{Hom}(\widetilde{K}_*^{\mathrm{TF}}(S^*),L)$ that $\chi$ is a map of Lie algebras.

\begin{lemma} \label{degh} For any space $A$, there is a commuting diagram
\begin{center}
\begin{tabular}{c}
\xymatrix{
\mathscr{B}(\pi_*(A)) \ar[r]^(.33){\mathrm{deg}'} \ar[d]_{h'} & \mathrm{Hom}(\widetilde{K}_*^{\mathrm{TF}}(S^*),L(\widetilde{K}_*^{\mathrm{TF}}(A))) \ar[dl]^{\chi} \\
L(\widetilde{K}_*^{\mathrm{TF}}(A)).
}  
\end{tabular}
\end{center}
\end{lemma}

\begin{proof} Commutativity of the diagram \begin{center}
\begin{tabular}{c}
\xymatrix{
\pi_*(A) \ar[r]^(.33){\mathrm{deg}} \ar[d]_{h} & \mathrm{Hom}(\widetilde{K}_*^{\mathrm{TF}}(S^*),\widetilde{K}_*^{\mathrm{TF}}(A)) \ar[dl]^{\chi} \\
\widetilde{K}_*^{\mathrm{TF}}(A).
}  
\end{tabular}
\end{center} follows from the definitions. Commutativity of the diagram from the lemma statement then follows from the definition of $\mathscr{B}(\pi_*(A))$, since $\chi$ respects brackets. \end{proof}

The third and fourth lemmas allow us to make the statement of Theorem \ref{K-detection} an entirely $p$-local one.

\begin{lemma} \label{liftLemma} Let $p$ be a prime, let $X$ be a simply connected $CW$-complex, and let $f: S^{q_1 + 1} \vee S^{q_2 + 1} \longrightarrow X_{(p)}$ be any map. There exists a map $\widetilde{f}: S^{q_1 + 1} \vee S^{q_2 + 1} \longrightarrow X$, and a map $\varphi : S^{q_1 + 1} \vee S^{q_2 + 1} \longrightarrow S^{q_1 + 1} \vee S^{q_2 + 1}$ which is a homotopy equivalence after $p$-localization, making the diagram \begin{center}
\begin{tabular}{c}
\xymatrix{ S^{q_1 + 1} \vee S^{q_2 + 1} \ar^(.65){\widetilde{f}}[r] \ar_{\varphi}[d] & X \ar[d] \\
\ar^(.65){f}[r] S^{q_1 + 1} \vee S^{q_2 + 1}  & X_{(p)}
}
\end{tabular}
\end{center} commute, where the vertical arrow is the localization.\end{lemma}

\begin{proof} Write $f = f_1 \vee f_2$, for maps $f_i:S^{q_i+1} \longrightarrow X$. On homotopy groups, the localizing map may be identified with the tensor map $$\pi_*(X) \longrightarrow \pi_*(X) \otimes \mathbb{Z}_{(p)} \cong \pi_*(X_{(p)}).$$ This implies that there exist integers $u_1$ and $u_2$, not divisible by $p$, such that $u_i f_i$ lifts to a map $\widetilde{f}_i : S^{q_i+1} \longrightarrow X$. Setting $\widetilde{f} = \widetilde{f}_1 \vee \widetilde{f}_2$, and letting $\varphi$ be the degree $u_i$ map on each wedge summand, we obtain the desired diagram. \end{proof}

\begin{lemma} \label{refereeLemma} Let $p$ be prime, and let $X$ and $Y$ be connected $CW$-complexes. If $\mu: S^{q_1+1} \vee S^{q_2+1} \longrightarrow \Sigma X$ induces an injection on $\widetilde{K}_*(\phantom{A}) \otimes \Zp$, and there is a homotopy equivalence of $p$-localizations $\Sigma Y_{(p)} \simeq \Sigma X_{(p)}$, then there exists a map $\mu' : S^{q_1+1} \vee S^{q_2+1} \longrightarrow \Sigma Y$ which also induces an injection on $\widetilde{K}_*(\phantom{A}) \otimes \Zp$. \end{lemma}

\begin{proof} For any space $A$, the localizing map $A \longrightarrow A_{(p)}$ induces an isomorphism on $K_*(\phantom{A}) \otimes \Zp$ \cite{Mislin2}. It therefore suffices to show that there exists maps $\mu'$ and $\varphi$ making the following diagram commute, with the localization $\varphi_{(p)}$ being a homotopy equivalence.  \begin{center}
\begin{tabular}{c}
\xymatrix{
S^{q_1+1} \vee S^{q_2+1} \ar@{.>}_{\varphi}[d]  \ar@{.>}^(.65){\mu'}[r] & \Sigma Y \ar[r] & \Sigma Y_{(p)} \\
S^{q_1+1} \vee S^{q_2+1} \ar^(.65){\mu}[r] & \Sigma X \ar[r] & \Sigma X_{(p)}. \ar_{\simeq}[u]
}
\end{tabular}
\end{center} Such maps exist by Lemma \ref{liftLemma}. \end{proof}

We are now ready to prove Theorem \ref{K-detection}.

\begin{proof}[Proof of Theorem \ref{K-detection}] Let $\mu = \mu_1 \vee \mu_2$, with adjoint $\overline{\mu} : S^{q_1} \vee S^{q_2} \to \Omega \Sigma X$. Let $f=f_{p,ck} \in \pi_{2ck(p-1)+2}(S^3)$. Consider the diagram of Construction \ref{kConstruction}, with $A=S^{q_1} \vee S^{q_2}$, $\qmax = \mathrm{max}(q_1,q_2)$, $\qmin = \mathrm{min}(q_1,q_2)$, and $\nu = \overline{\mu}$. We have such a diagram for each $k \in \mathbb{Z}^+$: \begin{center}
\begin{tabular}{c}
\xymatrix@C=0em{ \bigcup_{N = k\qmin}^{k \qmax} \mathscr{B}_N^k(\pi_*(S^{q_1} \vee S^{q_2})) \ar[r]^{f^*_\Sigma \circ \widetilde{\Phi_{\overline{\mu}}^\pi}} \ar[d]_{ \tau_p \circ \mathrm{deg}'} & \bigoplus_{N = k \qmin}^{k \qmax} \pi_{N+2ck(p-1)-1}(J_s(X)) \ar[d]^{\de} \\
\bigoplus_{N = k\qmin}^{k \qmax} I^k_N(S^{q_1} \vee S^{q_2}) \ar[r] & \bigoplus_{N = k \qmin}^{k \qmax} \mathrm{Ext}_{\psi \mathrm{-Mod}}(\widetilde{K}^*_{\mathrm{TF}}(J_s(X)),\widetilde{K}^*_{\mathrm{TF}}(S^{N+2ck(p-1)})).
}
\end{tabular}
\end{center} By assumption, $\mu^* \otimes \Zp: \widetilde{K}^*(\Sigma X) \otimes \Zp \to \widetilde{K}^*(S^{q_1 +1} \vee S^{q_2 +1}) \otimes \Zp$
is a surjection. Since $\widetilde{K}^*(S^{q_1 +1} \vee S^{q_2 +1})$ is torsion-free, Lemma \ref{useOfUCT} then implies that $$\mu_* \otimes \Zp : \widetilde{K}_*^{\mathrm{TF}}(S^{q_1 +1} \vee S^{q_2 +1}) \otimes \Zp \to \widetilde{K}_*^{\mathrm{TF}}(\Sigma X) \otimes \Zp$$ is an injection. By Lemma \ref{refereeLemma}, we may assume without loss of generality that $X$ has the integral homotopy type of a finite $CW$-complex. Thus, by Corollary \ref{kDgrm}, we may fix $c$ such that for large enough $k$, $\theta_p(f) \circ \mathscr{U}' \circ (\widetilde{\Phi_{\overline{\mu}}^K} \otimes \Zp)$ is an injection.


The Hurewicz map $h$ is a surjection $\pi_*(S^{q_1} \vee S^{q_2}) \to \widetilde{K}_*^{\mathrm{TF}}(S^{q_1} \vee S^{q_2})$, so the submodule generated by the image of the map $h' : \mathscr{B} (\pi_*(S^{q_1} \vee S^{q_2})) \to L(\widetilde{K}_*^{\mathrm{TF}}(S^{q_1} \vee S^{q_2}))$ of Lemma \ref{degh} contains the submodule generated by $\widetilde{K}_*^{\mathrm{TF}}(S^{q_1} \vee S^{q_2})$ under the bracket operation. In particular, it contains the weight $k$ component $L^k(\widetilde{K}_*^{\mathrm{TF}}(S^{q_1} \vee S^{q_2}))$ for each $k$. By Theorem \ref{ungradedWitt}, $\dim_{\mathbb{Z}} (L^k(\widetilde{K}_*^{\mathrm{TF}}(S^{q_1} \vee S^{q_2}))) = W_2(k)$. Note that $L^k(\widetilde{K}_*^{\mathrm{TF}}(S^{q_1} \vee S^{q_2})) = \bigoplus_{ N = k\qmin}^{k \qmax} L^k(\widetilde{K}_*^{\mathrm{TF}}(S^{q_1} \vee S^{q_2}))$.

\begin{sloppypar} It then follows from Lemma \ref{degh} that $\mathrm{dim}_{\mathbb{Z}/p}(\bigoplus_{N = k\qmin}^{k \qmax} I^k_N(S^{q_1} \vee S^{q_2})) \geq W_2(k)$. Since $\theta_\Sigma^p(f) \circ \mathscr{U}' \circ (\widetilde{\Phi_{\overline{\mu}}^K} \otimes \Zp)$ is an injection for large enough $k$, it follows that the dimension of $\de(\bigoplus_{ N = k\qmin}^{k \qmax} \pi_{N+2ck(p-1)-1}(J_s(X)))$ is at least $W_2(k)$. By Corollary \ref{RestrictDgrm} $(i_s)_*$ is an injection, so the dimension of $(i_s)_*(\bigoplus_{N = k \qmin}^{k \qmax} \pi_{N+2ck(p-1)-1}(J_s(X))) \subset \bigoplus_{N = k\qmin}^{k \qmax} \pi_{N+2ck(p-1)-1}(\Omega \Sigma X)$ is also at least $W_2(k)$.  \end{sloppypar} Thus, $\Sigma X$ satisfies the hypotheses of Lemma \ref{linearimplieshype} with $a=2c(p-1)+ \qmax = 2c(p-1)+ \mathrm{max}(q_1,q_2)$ and $b=0$, and hence is $p$-hyperbolic. \end{proof}

\printbibliography


\end{document}